\documentclass[12pt]{amsart}
\usepackage[dvipsnames]{xcolor}
\usepackage{tikz}
\usetikzlibrary{arrows.meta}
\usepackage{graphicx}
\usepackage{amssymb, amsthm, amsmath}
\usepackage{etoolbox}
\apptocmd{\sloppy}{\hbadness 10000\relax}{}{}
\apptocmd{\sloppy}{\vbadness 10000\relax}{}{}
\usepackage[justification=justified]{caption}
\usepackage{hyperref}
\usepackage[letterpaper,margin=1.1in]{geometry}

\usepackage{booktabs}
\usepackage{array}
\usepackage{longtable}
\usepackage{enumitem}

\definecolor{hc}{RGB}{0,128,128}
\definecolor{vc}{RGB}{142,69,133}

\newtheorem{introthm}{Theorem}

\newtheorem{theorem}{Theorem}[section]
\newtheorem{proposition}[theorem]{Proposition}
\newtheorem{lemma}[theorem]{Lemma}
\newtheorem{corollary}[theorem]{Corollary}

\theoremstyle{definition}
\newtheorem{definition}[theorem]{Definition}

\theoremstyle{remark}

\numberwithin{equation}{section}
\numberwithin{figure}{section}
\numberwithin{table}{section}

\allowdisplaybreaks[1]
\newcommand{\st}[1]{\mathrm{#1}}

\title[Axis-parallel square packing]{Axis-parallel square packing with an exact formula for optimal packing of seven squares}
\date{November 10, 2026}
\author{Matthew Badger}
\address{Department of Mathematics\\ University of Connecticut\\ Storrs, CT 06269-1009, USA}
\email{matthew.badger@uconn.edu}
\begin{document}
\begin{abstract}
How large a square is needed to hold axis-parallel squares of prescribed
side lengths? We give exact formulas for the optimal side through seven
squares, with explicit constructions and geometric lower-bound proofs.
For every $n\ge6$, the optimum for $n$ squares is a sum of at most $n-3$
of the prescribed side lengths. This bound is sharp on an open set of side
lengths and yields the optimal Euclidean Lipschitz constant $\sqrt{n-3}$
of the optimum as a function of the $n$ side lengths.
\end{abstract}

\maketitle

\setcounter{tocdepth}{1}
\tableofcontents

\section{Introduction}\label{sec:introduction}
For every sequence $s=(s_i)_{i=0}^\infty$ of real
numbers with $s_0\geq s_1 \geq \cdots\geq 0$, let $L^*(s)$ denote the smallest side length of an axis-parallel square that contains
axis-parallel squares of side lengths $s_0,s_1,\ldots$ with pairwise disjoint
interiors. Rotations are not allowed and sides of length zero signify
absent squares. By monotonicity and additivity of area and by a theorem of Moon and
Moser~\cite{MoonMoser67}, $L^*(s)$ is finite exactly when $s\in\ell_2$,
and $\|s\|_2\le L^*(s)\le\sqrt2\,\|s\|_2$. Hence $L^*$ is defined on the
cone $\ell_2^\downarrow$ of nonincreasing nonnegative sequences in $\ell_2$.
We identify a finite nonincreasing nonnegative tuple with the sequence in
$\ell_2^\downarrow$ obtained by appending zeros.
For all $n\ge1$ and $s\in\ell_2^\downarrow$, we write
\[
L_n^*(s)=L^*(s_0,\ldots,s_{n-1},0,0,\ldots)
\]
for the optimum for the $n$ largest squares.
Section~\ref{sec:function} shows that $L^*$ is continuous and $L^*_n$ is piecewise linear for each $n$.

\begin{figure}[t]\centering
\begin{minipage}[t]{0.32\textwidth}\centering
\begin{tikzpicture}[x=0.2cm,y=0.2cm]
\useasboundingbox (0,-2.295) rectangle (17.34,17);
\filldraw[fill=black!7,draw=black,line width=.35pt] (0,0) rectangle (9,9);
\node[font=\scriptsize,inner sep=0pt] at (4.5,4.5) {$0$};
\filldraw[fill=black!7,draw=black,line width=.35pt] (9,0) rectangle (17,8);
\node[font=\scriptsize,inner sep=0pt] at (13,4) {$1$};
\filldraw[fill=black!7,draw=black,line width=.35pt] (0,10) rectangle (7,17);
\node[font=\scriptsize,inner sep=0pt] at (3.5,13.5) {$2$};
\filldraw[fill=black!7,draw=black,line width=.35pt] (11,11) rectangle (17,17);
\node[font=\scriptsize,inner sep=0pt] at (14,14) {$3$};
\draw[line width=.7pt] (0,0) rectangle (17,17);
\filldraw[fill=hc!25,draw=hc!70!black,line width=.25pt] (0,-1.955) rectangle (9,-0.68);
\node[font=\tiny,inner sep=0pt] at (4.5,-1.3175) {$0$};
\filldraw[fill=hc!25,draw=hc!70!black,line width=.25pt] (9,-1.955) rectangle (17,-0.68);
\node[font=\tiny,inner sep=0pt] at (13,-1.3175) {$1$};
\end{tikzpicture}\par
\smallskip{\footnotesize $L_4^*=17$\\ $s=(9,8,7,6)$}
\end{minipage}\hfill
\begin{minipage}[t]{0.32\textwidth}\centering
\begin{tikzpicture}[x=0.1889cm,y=0.1889cm]
\useasboundingbox (0,-2.43) rectangle (18.36,18);
\filldraw[fill=black!7,draw=black,line width=.35pt] (0,0) rectangle (9,9);
\node[font=\scriptsize,inner sep=0pt] at (4.5,4.5) {$0$};
\filldraw[fill=black!7,draw=black,line width=.35pt] (10,0) rectangle (18,8);
\node[font=\scriptsize,inner sep=0pt] at (14,4) {$1$};
\filldraw[fill=black!7,draw=black,line width=.35pt] (0,11) rectangle (7,18);
\node[font=\scriptsize,inner sep=0pt] at (3.5,14.5) {$2$};
\filldraw[fill=black!7,draw=black,line width=.35pt] (7,12) rectangle (13,18);
\node[font=\scriptsize,inner sep=0pt] at (10,15) {$3$};
\filldraw[fill=black!7,draw=black,line width=.35pt] (13,13) rectangle (18,18);
\node[font=\scriptsize,inner sep=0pt] at (15.5,15.5) {$4$};
\draw[line width=.7pt] (0,0) rectangle (18,18);
\filldraw[fill=hc!25,draw=hc!70!black,line width=.25pt] (0,-2.07) rectangle (7,-0.72);
\node[font=\tiny,inner sep=0pt] at (3.5,-1.395) {$2$};
\filldraw[fill=hc!25,draw=hc!70!black,line width=.25pt] (7,-2.07) rectangle (13,-0.72);
\node[font=\tiny,inner sep=0pt] at (10,-1.395) {$3$};
\filldraw[fill=hc!25,draw=hc!70!black,line width=.25pt] (13,-2.07) rectangle (18,-0.72);
\node[font=\tiny,inner sep=0pt] at (15.5,-1.395) {$4$};
\end{tikzpicture}\par
\smallskip{\footnotesize $L_5^*=18$\\ $s=(9,8,7,6,5)$}
\end{minipage}\hfill
\begin{minipage}[t]{0.32\textwidth}\centering
\begin{tikzpicture}[x=0.1214cm,y=0.1214cm]
\useasboundingbox (0,-0.56) rectangle (31.78,28);
\filldraw[fill=black!7,draw=black,line width=.35pt] (0,0) rectangle (14,14);
\node[font=\scriptsize,inner sep=0pt] at (7,7) {$0$};
\filldraw[fill=black!7,draw=black,line width=.35pt] (15,0) rectangle (28,13);
\node[font=\scriptsize,inner sep=0pt] at (21.5,6.5) {$1$};
\filldraw[fill=black!7,draw=black,line width=.35pt] (0,18) rectangle (10,28);
\node[font=\scriptsize,inner sep=0pt] at (5,23) {$2$};
\filldraw[fill=black!7,draw=black,line width=.35pt] (10,19) rectangle (19,28);
\node[font=\scriptsize,inner sep=0pt] at (14.5,23.5) {$3$};
\filldraw[fill=black!7,draw=black,line width=.35pt] (20,20) rectangle (28,28);
\node[font=\scriptsize,inner sep=0pt] at (24,24) {$4$};
\filldraw[fill=black!7,draw=black,line width=.35pt] (21,13) rectangle (28,20);
\node[font=\scriptsize,inner sep=0pt] at (24.5,16.5) {$5$};
\draw[line width=.7pt] (0,0) rectangle (28,28);
\filldraw[fill=vc!25,draw=vc!70!black,line width=.25pt] (29.12,0) rectangle (31.22,13);
\node[font=\tiny,inner sep=0pt] at (30.17,6.5) {$1$};
\filldraw[fill=vc!25,draw=vc!70!black,line width=.25pt] (29.12,13) rectangle (31.22,20);
\node[font=\tiny,inner sep=0pt] at (30.17,16.5) {$5$};
\filldraw[fill=vc!25,draw=vc!70!black,line width=.25pt] (29.12,20) rectangle (31.22,28);
\node[font=\tiny,inner sep=0pt] at (30.17,24) {$4$};
\end{tikzpicture}\par
\smallskip{\footnotesize $L_6^*=28$\\ $s=(14,13,10,9,8,7)$}
\end{minipage}
\caption{Optimal packings of four, five, and six squares. In the second and third
panels, $L_5^*=L_4^*+1$ and $L_6^*=L_5^*+1$. A teal bar records the horizontal
projections of a chain whose side sum is the side length of the square container; a purple bar
on the right similarly records a vertical chain.}
\label{fig:small}
\end{figure}

\subsection{Exact formulas}
The exact formulas through seven squares are as follows.
\begin{introthm}\label{thm:intro-formulas}
For every $s\in\ell_2^\downarrow$, we have $L_1^*(s)=s_0$ and
$L_n^*(s)=s_0+s_1$ for $2\le n\le4$. For five and six squares,
\begin{align}
L_5^*(s)&=\max(s_0+s_1,\ s_2+s_3+s_4),\label{eq:intro-five}\\
L_6^*(s)&=\max(s_0+s_1,\ s_2+s_3+s_4,\ s_1+s_4+s_5).
\label{eq:intro-six}
\end{align}
For seven squares,
\begin{equation}\label{eq:intro-seven}
\begin{aligned}
L_7^*(s)=\max\bigl(L_6^*(s),\min\{&\max(s_3+s_4+s_5+s_6,\ s_1+s_2+s_5),\\
&\max(s_0+s_5+s_6,\ s_1+s_3+s_6),\ s_1+s_4+s_5+s_6\}\bigr).
\end{aligned}
\end{equation}
\end{introthm}

A horizontal chain is a family of
squares whose horizontal projections have disjoint interiors; vertical
chains are defined analogously. The side sum of a collection of squares is
the sum of their side lengths. In any packing of axis-parallel squares, the
side sum of a chain is at most the container side. The formulas for at most four squares are
elementary. The five-square
formula~\eqref{eq:intro-five} is implicit in work of Lan, D\'osa, Han,
Zhou, and Benk\H{o}~\cite{LanDosa13} on packing as many squares as
possible into a unit square. In proving that a greedy shelf algorithm is
optimal whenever at most five squares fit \cite[Theorem~1]{LanDosa13},
they show that five squares fit into a unit square if and only if the two
largest sides sum to at most one and the three smallest sides sum to at
most one; the second panel of Figure~\ref{fig:small} illustrates this. We include a short proof of~\eqref{eq:intro-five}, since its
chain argument is the starting point for six and seven squares.

The growth of complexity in the formula for $L^*_n$ from $n=6$ to $n=7$
is reflected in the minimum size of a strategy family attaining the optimum
for every nonincreasing side vector, which rises from one for six squares to
three for seven squares; see Figure~\ref{fig:seven}. In fact, $L_7^*$ cannot
be expressed as the minimum of two convex functions. In a companion
paper~\cite{BadgerEight}, we prove that $L_8^*$ is the maximum of $L_7^*$
and a minimum of twenty-three maxima of side sums. The proof there uses the methods developed here. At this time, a formula for $L^*_9$ is unknown.

\begin{figure}[t]\centering
\begin{minipage}[t]{0.32\textwidth}\centering
\begin{tikzpicture}[x=0.0895cm,y=0.0895cm]
\useasboundingbox (0,-5.13) rectangle (43.13,38);
\filldraw[fill=black!7,draw=black,line width=.35pt] (0,0) rectangle (22,22);
\node[font=\scriptsize,inner sep=0pt] at (11,11) {$0$};
\filldraw[fill=black!7,draw=black,line width=.35pt] (23,0) rectangle (38,15);
\node[font=\scriptsize,inner sep=0pt] at (30.5,7.5) {$1$};
\filldraw[fill=black!7,draw=black,line width=.35pt] (24,15) rectangle (38,29);
\node[font=\scriptsize,inner sep=0pt] at (31,22) {$2$};
\filldraw[fill=black!7,draw=black,line width=.35pt] (0,27) rectangle (11,38);
\node[font=\scriptsize,inner sep=0pt] at (5.5,32.5) {$3$};
\filldraw[fill=black!7,draw=black,line width=.35pt] (11,28) rectangle (21,38);
\node[font=\scriptsize,inner sep=0pt] at (16,33) {$4$};
\filldraw[fill=black!7,draw=black,line width=.35pt] (29,29) rectangle (38,38);
\node[font=\scriptsize,inner sep=0pt] at (33.5,33.5) {$5$};
\filldraw[fill=black!7,draw=black,line width=.35pt] (21,30) rectangle (29,38);
\node[font=\scriptsize,inner sep=0pt] at (25,34) {$6$};
\draw[line width=.7pt] (0,0) rectangle (38,38);
\filldraw[fill=hc!25,draw=hc!70!black,line width=.25pt] (0,-4.37) rectangle (11,-1.52);
\node[font=\tiny,inner sep=0pt] at (5.5,-2.945) {$3$};
\filldraw[fill=hc!25,draw=hc!70!black,line width=.25pt] (11,-4.37) rectangle (21,-1.52);
\node[font=\tiny,inner sep=0pt] at (16,-2.945) {$4$};
\filldraw[fill=hc!25,draw=hc!70!black,line width=.25pt] (21,-4.37) rectangle (29,-1.52);
\node[font=\tiny,inner sep=0pt] at (25,-2.945) {$6$};
\filldraw[fill=hc!25,draw=hc!70!black,line width=.25pt] (29,-4.37) rectangle (38,-1.52);
\node[font=\tiny,inner sep=0pt] at (33.5,-2.945) {$5$};
\filldraw[fill=vc!25,draw=vc!70!black,line width=.25pt] (39.52,0) rectangle (42.37,15);
\node[font=\tiny,inner sep=0pt] at (40.945,7.5) {$1$};
\filldraw[fill=vc!25,draw=vc!70!black,line width=.25pt] (39.52,15) rectangle (42.37,29);
\node[font=\tiny,inner sep=0pt] at (40.945,22) {$2$};
\filldraw[fill=vc!25,draw=vc!70!black,line width=.25pt] (39.52,29) rectangle (42.37,38);
\node[font=\tiny,inner sep=0pt] at (40.945,33.5) {$5$};
\end{tikzpicture}\par
\smallskip{\footnotesize $\st{top}$,\quad $L_7^*=38$\\ $s=(22,15,14,11,10,9,8)$}
\end{minipage}\hfill
\begin{minipage}[t]{0.32\textwidth}\centering
\begin{tikzpicture}[x=0.0944cm,y=0.0944cm]
\useasboundingbox (0,-4.86) rectangle (40.86,36);
\filldraw[fill=black!7,draw=black,line width=.35pt] (0,0) rectangle (19,19);
\node[font=\scriptsize,inner sep=0pt] at (9.5,9.5) {$0$};
\filldraw[fill=black!7,draw=black,line width=.35pt] (20,0) rectangle (36,16);
\node[font=\scriptsize,inner sep=0pt] at (28,8) {$1$};
\filldraw[fill=black!7,draw=black,line width=.35pt] (0,23) rectangle (13,36);
\node[font=\scriptsize,inner sep=0pt] at (6.5,29.5) {$2$};
\filldraw[fill=black!7,draw=black,line width=.35pt] (13,24) rectangle (25,36);
\node[font=\scriptsize,inner sep=0pt] at (19,30) {$3$};
\filldraw[fill=black!7,draw=black,line width=.35pt] (26,26) rectangle (36,36);
\node[font=\scriptsize,inner sep=0pt] at (31,31) {$4$};
\filldraw[fill=black!7,draw=black,line width=.35pt] (27,16) rectangle (36,25);
\node[font=\scriptsize,inner sep=0pt] at (31.5,20.5) {$5$};
\filldraw[fill=black!7,draw=black,line width=.35pt] (19,16) rectangle (27,24);
\node[font=\scriptsize,inner sep=0pt] at (23,20) {$6$};
\draw[line width=.7pt] (0,0) rectangle (36,36);
\filldraw[fill=hc!25,draw=hc!70!black,line width=.25pt] (0,-4.14) rectangle (19,-1.44);
\node[font=\tiny,inner sep=0pt] at (9.5,-2.79) {$0$};
\filldraw[fill=hc!25,draw=hc!70!black,line width=.25pt] (19,-4.14) rectangle (27,-1.44);
\node[font=\tiny,inner sep=0pt] at (23,-2.79) {$6$};
\filldraw[fill=hc!25,draw=hc!70!black,line width=.25pt] (27,-4.14) rectangle (36,-1.44);
\node[font=\tiny,inner sep=0pt] at (31.5,-2.79) {$5$};
\filldraw[fill=vc!25,draw=vc!70!black,line width=.25pt] (37.44,0) rectangle (40.14,16);
\node[font=\tiny,inner sep=0pt] at (38.79,8) {$1$};
\filldraw[fill=vc!25,draw=vc!70!black,line width=.25pt] (37.44,16) rectangle (40.14,24);
\node[font=\tiny,inner sep=0pt] at (38.79,20) {$6$};
\filldraw[fill=vc!25,draw=vc!70!black,line width=.25pt] (37.44,24) rectangle (40.14,36);
\node[font=\tiny,inner sep=0pt] at (38.79,30) {$3$};
\end{tikzpicture}\par
\smallskip{\footnotesize $\st{center}$,\quad $L_7^*=36$\\ $s=(19,16,13,12,10,9,8)$}
\end{minipage}\hfill
\begin{minipage}[t]{0.32\textwidth}\centering
\begin{tikzpicture}[x=0.0654cm,y=0.0654cm]
\useasboundingbox (0,-1.04) rectangle (59.02,52);
\filldraw[fill=black!7,draw=black,line width=.35pt] (0,0) rectangle (26,26);
\node[font=\scriptsize,inner sep=0pt] at (13,13) {$0$};
\filldraw[fill=black!7,draw=black,line width=.35pt] (27,0) rectangle (52,25);
\node[font=\scriptsize,inner sep=0pt] at (39.5,12.5) {$1$};
\filldraw[fill=black!7,draw=black,line width=.35pt] (0,31) rectangle (21,52);
\node[font=\scriptsize,inner sep=0pt] at (10.5,41.5) {$2$};
\filldraw[fill=black!7,draw=black,line width=.35pt] (21,32) rectangle (41,52);
\node[font=\scriptsize,inner sep=0pt] at (31,42) {$3$};
\filldraw[fill=black!7,draw=black,line width=.35pt] (42,42) rectangle (52,52);
\node[font=\scriptsize,inner sep=0pt] at (47,47) {$4$};
\filldraw[fill=black!7,draw=black,line width=.35pt] (43,25) rectangle (52,34);
\node[font=\scriptsize,inner sep=0pt] at (47.5,29.5) {$5$};
\filldraw[fill=black!7,draw=black,line width=.35pt] (44,34) rectangle (52,42);
\node[font=\scriptsize,inner sep=0pt] at (48,38) {$6$};
\draw[line width=.7pt] (0,0) rectangle (52,52);
\filldraw[fill=vc!25,draw=vc!70!black,line width=.25pt] (54.08,0) rectangle (57.98,25);
\node[font=\tiny,inner sep=0pt] at (56.03,12.5) {$1$};
\filldraw[fill=vc!25,draw=vc!70!black,line width=.25pt] (54.08,25) rectangle (57.98,34);
\node[font=\tiny,inner sep=0pt] at (56.03,29.5) {$5$};
\filldraw[fill=vc!25,draw=vc!70!black,line width=.25pt] (54.08,34) rectangle (57.98,42);
\node[font=\tiny,inner sep=0pt] at (56.03,38) {$6$};
\filldraw[fill=vc!25,draw=vc!70!black,line width=.25pt] (54.08,42) rectangle (57.98,52);
\node[font=\tiny,inner sep=0pt] at (56.03,47) {$4$};
\end{tikzpicture}\par
\smallskip{\footnotesize $\st{right}$,\quad $L_7^*=52$\\ $s=(26,25,21,20,10,9,8)$}
\end{minipage}
\caption{The ``top'', ``center'', and ``right'' packing strategies for seven squares at side
vectors where the indicated strategy cost is uniquely optimal among the
three strategies. The
projection bars record the chains that determine the container side. Each example has $L^*_7(s)=L^*_6(s)+1$. Labels and colors are as in
Figure~\ref{fig:small}.}
\label{fig:seven}
\end{figure}

\subsection{How many sides are needed to express the optimum?}
Section~\ref{sec:function} shows that $L_n^*(s)$ is a
piecewise linear function of $(s_0,\ldots,s_{n-1})$ whose values are given by
a sum of some of the side lengths. The next theorem identifies the largest
possible minimum number of sides needed to express the optimum as a subset
sum.
\begin{introthm}\label{thm:intro-longest}
For every $n\ge6$ and every $s\in\ell_2^\downarrow$, there is a set
$I\subseteq\{0,\ldots,n-1\}$ with $|I|\le n-3$ and $L_n^*(s)=\sum_{i\in I}s_i$.
This bound is sharp: there is a nonempty open set
$U\subseteq\{(s_0,\ldots,s_{n-1})\in\mathbb R^n:s_0>s_1>\cdots>s_{n-1}>0\}$
such that
\begin{equation}\label{longest-sum}
L_n^*(s)=s_1+\sum_{i=4}^{n-1}s_i\quad\text{for all }s\in\ell_2^\downarrow\text{ with }(s_0,\ldots,s_{n-1})\in U,
\end{equation}
and no subset sum with fewer than $n-3$ entries equals this value.
Moreover, $L_n^*(s)$ is given by a finite minimum-of-maxima formula in
$s_0,\ldots,s_{n-1}$ involving only subset sums of cardinality at most
$n-3$, and
\begin{equation}\label{sharp-lipschitz}
|L_n^*(s)-L_n^*(t)|\le\sqrt{n-3}\,\|s-t\|_2\quad\text{for all }s,t\in\ell_2^\downarrow,
\end{equation}
where the constant $\sqrt{n-3}$ cannot be decreased.
\end{introthm}

\subsection{Motivation and related packing problems}
The author's interest in $L^*$ arose from joint work with Raanan Schul on
Lipschitz parameterizations and rectifiable doubling measures
\cite{BadgerSchul25}. In these constructions a tree organizes
successive pieces of a target set, and squares assigned to the children of
a vertex must fit into a square assigned to their parent. The supremum
metric is natural on the parameter plane because its balls are
axis-parallel squares. Thus, $L_n^*(s)$ is the exact local cost of the packing
operation. Replacing an upper estimate by this optimal cost avoids a loss
at a step that is repeated throughout a multiscale construction.
One useful estimate from \cite{BadgerSchul25} is
\[
L_n^*(s)\le\sum_{j^2<n}s_{j^2}=s_0+s_1+s_4+s_9+\cdots.
\]
It is obtained by successive L-shaped layers and gives
$\lceil\sqrt n\rceil s$ for $n$ equal sides of length $s$.
The present paper studies the optimal function independently of those
applications.

Classical packing theorems often give a guarantee in terms of total area,
or optimize the area of a bounding rectangle. The theorem of Moon and
Moser~\cite{MoonMoser67}, and its extensions by Meir and Moser~\cite{MeirMoser68},
are fundamental examples; see also Hougardy~\cite{Hougardy11} for packing
squares of prescribed total area into a rectangle. For a fixed finite side
vector, minimizing the side of a square container is of course equivalent
to maximizing the occupied fraction $\sum_i s_i^2/L^2$. Our aim is not a
universal or asymptotic density estimate, but the exact dependence of the
optimum on every prescribed side. Lower bounds for packing squares and rectangles into squares are analyzed by Caprara, Lodi, Martello, and
Monaci~\cite{Caprara06}.

Another well-known problem asks for the smallest square containing $n$
\emph{unit} squares with arbitrary rotations; Friedman~\cite{Friedman09}
surveys that problem and settles several of its small cases. With fixed
orientations, equal sides are elementary:
$L_n^*(1,\ldots,1)=\lceil\sqrt n\rceil$. Indeed, a grid gives the upper
bound, and in a packing in $[0,L]^2$ the integer parts of the lower left
corners are distinct points of $\{0,\ldots,\lfloor L\rfloor-1\}^2$. With
free orientations, five unit squares already fit into a square of side
$2+1/\sqrt2<3$, and that value is optimal~\cite{Friedman09}.

For rectangles of fixed orientations, Murata, Fujiyoshi, Nakatake, and
Kajitani \cite{Murata96} give a finite sequence-pair encoding and construct
placements by longest paths in two constraint graphs. Their result applies
to objectives nondecreasing in width and height, hence also to their maximum.
There are $(n!)^2$ sequence-pairs. Silvanus and
Vygen~\cite{SilvanusVygen20} represent a placement of $n$ rectangles by
assigning to each pair one of the relations west, south, east, or north,
and call a set of such representations complete if it represents every
placement without overlaps. They proved that the least size of a complete
set lies between
\[
\Omega\left(\frac{n!}{n^4}\left(4+2\sqrt2\right)^n\right)
\quad\text{and}\quad
O\left(\frac{n!}{n^6}\left(\frac{11+5\sqrt5}{2}\right)^n\right).
\]
In this paper, a complete family
of strategies (Definition~\ref{def:complete}) need only attain $L_n^*(s)$
for every $s$, rather than represent every placement. The upper bound of
Silvanus and Vygen applies to such families, but their lower bounds do
not. Indeed, the least sizes of complete families for six and seven
squares are $1$ and $3$ (Theorem~\ref{thm:min-counts}), and for eight
squares the least size is twenty-three~\cite{BadgerEight}.
Subsection~\ref{sec:strategies} gives the partial-order formulation suited
to our arguments.

When $n$ is part of the input, deciding whether prescribed squares fit
orthogonally into a prescribed square is strongly NP-complete
\cite{Leung90}. For each fixed $n$, Proposition~\ref{prop:strategy} gives a
finite formula for $L_n^*$; the questions are its size and structure.

The rest of the paper is organized as follows. Section~\ref{sec:function} develops general properties of $L^*$ and of
finite packing strategies. Section~\ref{sec:small} introduces essential terminology used in proofs throughout the paper and proves Theorem~\ref{thm:intro-formulas} for up to six
squares. Section~\ref{sec:longest} proves Theorem~\ref{thm:intro-longest}. In
Section~\ref{sec:seven}, we complete the proof of Theorem~\ref{thm:intro-formulas} by verifying the formula for seven squares. Finally, in Section~\ref{sec:minimality}, we prove that $L^*_7$ cannot be described with fewer than three packing strategies.
\section{The square-packing function}\label{sec:function}
\subsection{Basic properties}\label{sec:basic}
A packing of a finite or countably infinite nonnegative side vector $s$ in
$[0,L]^2$ assigns to every positive entry $s_i$ a square
$Q_i=[x_i,x_i+s_i]\times[y_i,y_i+s_i]\subseteq[0,L]^2$,
with distinct squares having disjoint interiors. In other words, for all $i$ with $s_i>0$,
\begin{equation}\label{eq:contain}
0\le x_i,y_i\le L-s_i
\end{equation}
and, for all distinct $i$ and $j$ with $s_i,s_j>0$, at least one of the following holds:
\begin{equation}\label{eq:sep}
x_i+s_i\le x_j,\quad x_j+s_j\le x_i,\quad
y_i+s_i\le y_j,\quad y_j+s_j\le y_i.
\end{equation}
Squares in a packing may have touching boundaries. We define $L^*(s)$ to be the infimum of the set of all $L\ge0$ such that there exists a packing of $s$ in $[0,L]^2$ in the sense above. In order to establish basic properties of $L^*$, we first state the shelf packing of Moon and Moser for squares in the form used below and include a proof for the reader's convenience.

\begin{figure}[b]\centering
\begin{minipage}[t]{0.6\textwidth}\centering
\begin{tikzpicture}[x=0.1853cm,y=0.1853cm]
\useasboundingbox (0,-2.7301) rectangle (27.4964,24.0665);
\draw[densely dashed,line width=.35pt,black!55] (22,0) rectangle (27,5);
\filldraw[fill=blue!14,draw=black,line width=.35pt] (0,0) rectangle (9,9);
\node[font=\scriptsize,inner sep=0pt] at (4.5,4.5) {$a$};
\filldraw[fill=blue!14,draw=black,line width=.35pt] (9,0) rectangle (16,7);
\filldraw[fill=blue!14,draw=black,line width=.35pt] (16,0) rectangle (22,6);
\filldraw[fill=orange!22,draw=black,line width=.35pt] (0,9) rectangle (5,14);
\node[font=\scriptsize,inner sep=0pt] at (2.5,11.5) {$u_{1}$};
\filldraw[fill=orange!22,draw=black,line width=.35pt] (9,9) rectangle (14,14);
\filldraw[fill=orange!22,draw=black,line width=.35pt] (14,9) rectangle (18,13);
\filldraw[fill=orange!22,draw=black,line width=.35pt] (18,9) rectangle (22,13);
\filldraw[fill=green!18,draw=black,line width=.35pt] (0,14) rectangle (3,17);
\node[font=\scriptsize,inner sep=0pt] at (1.5,15.5) {$u_{2}$};
\filldraw[fill=green!18,draw=black,line width=.35pt] (9,14) rectangle (12,17);
\filldraw[fill=green!18,draw=black,line width=.35pt] (12,14) rectangle (15,17);
\filldraw[fill=green!18,draw=black,line width=.35pt] (15,14) rectangle (18,17);
\filldraw[fill=green!18,draw=black,line width=.35pt] (18,14) rectangle (20,16);
\filldraw[fill=green!18,draw=black,line width=.35pt] (20,14) rectangle (22,16);
\filldraw[fill=green!18,draw=black,line width=.35pt] (22,14) rectangle (24,16);
\filldraw[fill=red!14,draw=black,line width=.35pt] (0,17) rectangle (2,19);
\node[font=\tiny,inner sep=0pt] at (1,18) {$u_{3}$};
\filldraw[fill=red!14,draw=black,line width=.35pt] (9,17) rectangle (11,19);
\filldraw[fill=red!14,draw=black,line width=.35pt] (11,17) rectangle (13,19);
\draw[densely dashed,line width=.3pt] (9,0) -- (9,24.0665);
\node[font=\tiny,inner sep=1pt,below] at (9,0) {$a$};
\draw[densely dashed,line width=.3pt] (0,9) -- (24.0665,9);
\draw[densely dashed,line width=.3pt] (0,14) -- (24.0665,14);
\draw[densely dashed,line width=.3pt] (0,17) -- (24.0665,17);
\draw[line width=.7pt] (0,0) rectangle (24.0665,24.0665);
\end{tikzpicture}\par
\smallskip{\footnotesize }
\end{minipage}
\caption{The shelf packing in the proof of Lemma~\ref{lem:moon-moser} for a
square of side $a=9$ and the sides $7,6,5,5,4,4,3,3,3,3,2,2,2,2,2,2$, of total
area $T=227$, in the square of side $b=9+\sqrt{227}$.
The square of side $a$ and the column squares $u_1,u_2,u_3$ lie
left of the vertical dashed line. The horizontal dashed lines mark the heights at which
the shelves begin, and each shelf has the color of its column square. The
dashed outline shows that $u_1$ does not fit at the end of the
first shelf.}
\label{fig:moon-moser}
\end{figure}

\begin{lemma}[{Moon--Moser~\cite[p.~108]{MoonMoser67}}]\label{lem:moon-moser}
Let $a\ge0$, and let $\mathcal F$ be a finite or countably infinite family
of squares, each of side at most $a$, with finite total area $T$. Then a square of
side $a$ and the squares of $\mathcal F$ have an axis-parallel packing in a
square of side $a+\sqrt T$. In particular, every finite or countably
infinite family of squares with total area $A$ has an axis-parallel packing
in a square of side $\sqrt{2A}$.
\end{lemma}
\begin{proof} The case $T=0$ is immediate. Otherwise arrange the squares
of $\mathcal F$ with positive side lengths in decreasing order
$t_1\ge t_2\ge\cdots$ and put $b=a+\sqrt T$. Each $t_i$ is at most $a$ and
at most $\sqrt T=b-a$. Place the square of side $a$ at the lower left
corner of $[0,b]\times[0,\infty)$. Reserve the left strip of width $a$
for a column of squares. To its right, place successive squares on
\emph{horizontal shelves}: rows of squares placed consecutively from
left to right with their bottom edges aligned. Begin the first shelf
at height zero. Whenever the next square does not fit in the available
width $b-a$, put it in the left column immediately above the preceding
column square and begin another shelf at the same height
(Figure~\ref{fig:moon-moser}). Every square fits on an empty shelf, and
every column square fits in the strip. Since the sides are decreasing and
at most $a$, every square on a shelf is no taller than its column square,
the first column square being the square of side $a$. Thus, the shelves do
not overlap.

Let $u_1,u_2,\ldots$ be the sides of the column squares after the first,
and $w_j$ the total width occupied by the shelf immediately preceding
$u_j$. Since $u_j$ did not fit, $w_j+u_j>b-a$. Every square on that shelf
has side at least $u_j$. Hence the total area of the shelf together with
$u_j^2$ exceeds $u_j(b-a)$. These groups of squares are disjoint. Therefore, for each
$m$ for which $u_m$ exists,
\[
(b-a)\sum_{j=1}^m u_j\le T.
\]
Whether the number of shelves is finite or infinite, their total height
is consequently at most
\[
a+\sum_j u_j\le a+\frac{T}{b-a}=b.
\]
Every square is placed at a finite stage, and all shelves lie within
the width $b-a$ and below height $b$. Therefore, all the squares fit
in $[0,b]^2$.

For the last statement, assume $A>0$. Since only finitely many sides exceed
any positive number, there is a largest side $a$. The first part, applied to
the remaining squares with $T=A-a^2$, gives a packing in a square of side
$a+\sqrt{A-a^2}\le\sqrt{2A}$, where the inequality follows from
$2A-(a+\sqrt{A-a^2})^2=(a-\sqrt{A-a^2})^2$.
\end{proof}

\begin{lemma}[Existence, attainment, and truncation]\label{lem:function}
A weakly decreasing nonnegative sequence admits a packing in some finite square
if and only if it belongs to $\ell_2$. For $s\in\ell_2^\downarrow$, the
infimum defining $L^*(s)$ is attained. The function is monotone under
coordinatewise increase, positively homogeneous, that is,
$L^*(\lambda s)=\lambda L^*(s)$ for all $\lambda\ge0$, and satisfies
\begin{gather}
\|s\|_2\le L^*(s)\le\sqrt2\,\|s\|_2,\label{eq:area-bound}\\
\lim_{n\to\infty} L_n^*(s)=\sup_{n} L_n^*(s)= L^*(s).\label{eq:truncation-limit}
\end{gather}
Moreover, for all $n\ge1$,
\begin{equation}\label{eq:tail-estimate}
0\le L^*(s)-L_n^*(s)
\le\left(\sum_{i=n}^{\infty}s_i^2\right)^{1/2}.
\end{equation}
\end{lemma}

\begin{proof}
If $s$ has a packing in $[0,L]^2$, countable additivity of area gives
$\sum_i s_i^2\le L^2$, since the square interiors are disjoint.
Hence $s\in\ell_2$, and taking the infimum over $L$ gives
$\|s\|_2\le L^*(s)$. Conversely, for every $s\in\ell_2^\downarrow$,
Lemma~\ref{lem:moon-moser} gives a packing in a square of side
$\sqrt2\|s\|_2$. This proves the first assertion
and~\eqref{eq:area-bound}.

If $s_i\le t_i$ for every $i$, shrinking the squares in a packing of
$t$ while keeping their lower left corners fixed gives a packing of
$s$ in the same container. Thus, $L^*(s)\le L^*(t)$. In particular,
$L_n^*(s)$ is nondecreasing in $n$ and bounded above by $L^*(s)$.
Dilating packings gives $L^*(\lambda s)\le\lambda L^*(s)$ for
$\lambda>0$, and applying the inverse dilation gives equality.
The case $\lambda=0$ follows from $L^*(0)=0$.

Let $B=\sup_nL_n^*(s)$, which is finite by~\eqref{eq:area-bound}.
For each $n\ge1$, choose a packing of the positive entries among
$s_0,\ldots,s_{n-1}$ in $[0,B+1/n]^2$. A diagonal argument gives
a subsequence along which the coordinates of every square with
$s_i>0$ converge. The conditions~\eqref{eq:contain}
and~\eqref{eq:sep} are closed in the coordinates and the container
side, so the limiting coordinates give a packing of $s$ in
$[0,B]^2$. Hence $L^*(s)\le B$. Since $B\le L^*(s)$, equality holds
and the infimum is attained. Monotonicity of $L_n^*(s)$ then
gives~\eqref{eq:truncation-limit}.

Finally, let $a=L_n^*(s)$. Since $s_i\le s_{n-1}\le a$ for
$i\ge n$, Lemma~\ref{lem:moon-moser} gives a packing of a square
of side $a$ together with squares of sides $s_n,s_{n+1},\ldots$
in a square of side $a+(\sum_{i\ge n}s_i^2)^{1/2}$.
Placing an optimal packing of $s_0,\ldots,s_{n-1}$ inside the
square of side $a$ proves the upper inequality
in~\eqref{eq:tail-estimate}. The lower inequality follows
from monotonicity.
\end{proof}

\subsection{Packing strategies and continuity}\label{sec:strategies}
The combinatorial framework below is closely related to the
sequence-pair and constraint-graph method for fixed-orientation rectangles
in \cite{Murata96}.

\begin{definition}[Strategies and complete families]\label{def:strategy}\label{def:complete}
A \emph{strategy} for $n$ squares is a pair
$\sigma=(\prec_{\mathrm h},\prec_{\mathrm v})$ of strict partial orders on
$\{0,\ldots,n-1\}$ such that every pair of distinct indices is comparable
in at least one order. The relations prescribe horizontal and vertical
separation. A \emph{chain} of $\sigma$ is a nonempty subset of $\{0,...,n-1\}$ totally
ordered by either relation. Write $\mathcal C(\sigma)$ for the set of all
chains of $\sigma$ and define the \emph{cost} of $\sigma$ at a vector
$v\in[0,\infty)^n$ by
$\ell_\sigma(v)=\max_{I\in\mathcal C(\sigma)}\sum_{i\in I}v_i$. For
$s\in\ell_2^\downarrow$, we write $\ell_\sigma(s)$ for its value at
$(s_0,\ldots,s_{n-1})$.
A family $\mathcal S$ of strategies is \emph{complete} if
$L_n^*(s)=\min_{\sigma\in\mathcal S}\ell_\sigma(s)$ for every
$s\in\ell_2^\downarrow$.
\end{definition}

For the next lemma, let $t=(t_0,\ldots,t_{n-1})$ be a vector
of nonnegative side lengths that need not be in decreasing order; we reserve
$s$ for nonincreasing sequences. A packing of $t$ has squares
$Q_i=[x_i,x_i+t_i]\times[y_i,y_i+t_i]$ for $t_i>0$.

A packing of $t$ \emph{follows} a strategy $\sigma$ if, for all $i,j$ with
$t_i,t_j>0$, we have $x_i+t_i\le x_j$ whenever $i\prec_{\mathrm h}j$ and
$y_i+t_i\le y_j$ whenever $i\prec_{\mathrm v}j$.

If $t_0,\ldots,t_{n-1}>0$, then a packing of $t$ determines a strategy
$\sigma$, which we call the \emph{strategy of the packing}: put
$i\prec_{\mathrm h}j$ when $x_i+t_i\le x_j$, and $i\prec_{\mathrm v}j$ when
$y_i+t_i\le y_j$. Positivity makes each relation irreflexive, the
inequalities make it transitive, and~\eqref{eq:sep} makes every pair
comparable in at least one of them. The packing follows $\sigma$.

\begin{lemma}[Criterion to pack following a strategy]\label{lem:strategy-side}
Let $\sigma$ be a strategy for $n$ squares and let $t_0,\ldots,t_{n-1}$ be
nonnegative, not necessarily in decreasing order. There is a packing of $t$ in $[0,L]^2$ following $\sigma$
if and only if $L\ge\ell_\sigma(t)$.
\end{lemma}
\begin{proof}
Let $I$ be a chain of $\prec_{\mathrm h}$, and list its members with
positive sides in increasing order. In a following packing, adding the
separation inequalities for consecutive members and the containment
inequalities gives $\sum_{i\in I}t_i\le L$. The vertical argument is
identical.

Conversely, for each $j$ let $x_j$ be the largest sum
$\sum_{i\in I\setminus\{j\}}t_i$ over chains $I$ of $\prec_{\mathrm h}$ with
greatest element $j$, and define $y_j$ using chains of $\prec_{\mathrm v}$.
Singleton chains give $x_j,y_j\ge0$, and
$x_j+t_j,y_j+t_j\le\ell_\sigma(t)$. If $i\prec_{\mathrm h}j$, extend a
maximizing chain ending at $i$ by $j$ to obtain $x_j\ge x_i+t_i$. Likewise
$y_j\ge y_i+t_i$ for $i\prec_{\mathrm v}j$. Since every pair is comparable
in at least one order, these coordinates give the desired packing.
\end{proof}

\begin{proposition}[Strategy representation]\label{prop:strategy}
For every $n\ge1$ and every $s\in\ell_2^\downarrow$,
\begin{equation}\label{eq:strategy-formula}
L_n^*(s)=\min_{\sigma\in\mathfrak S_n}\ell_\sigma(s)
=\min_{\sigma\in\mathfrak S_n}\max_{I\in\mathcal C(\sigma)}\sum_{i\in I}s_i,
\end{equation}
where $\mathfrak S_n$ is the finite set of all strategies on $n$ squares.
\end{proposition}

\begin{proof}
For every $\sigma\in\mathfrak S_n$, Lemma~\ref{lem:strategy-side}
gives a packing of $s_0,\ldots,s_{n-1}$ in a square of side
$\ell_\sigma(s)$. Thus,
$L_n^*(s)\le\min_{\sigma\in\mathfrak S_n}\ell_\sigma(s)$.
Conversely, choose an optimal packing of the positive entries among
$s_0,\ldots,s_{n-1}$, and take the horizontal and vertical separation
orders on their labels. Extend the horizontal order by placing the
labels of zero entries, in increasing index order, after all existing
labels. Leave these added labels incomparable in the vertical order.
The resulting pair of orders is a strategy $\sigma$ followed by the
packing. By Lemma~\ref{lem:strategy-side},
$\ell_\sigma(s)\le L_n^*(s)$, proving~\eqref{eq:strategy-formula}.
\end{proof}

We remark that each strategy cost is a maximum of finitely many linear subset sums,
and there are only finitely many strategies. The formula therefore
shows that $L_n^*$ is continuous, positively homogeneous, and
piecewise linear. At each input, the formula selects one of its
constituent subset sums.

\begin{lemma}[Continuity]\label{lem:continuity}
For every $n\ge1$ and $s,t\in\ell_2^\downarrow$,
\begin{equation}\label{eq:finite-lipschitz}
|L_n^*(s)-L_n^*(t)|\le\sum_{i<n}|s_i-t_i|\le\sqrt n\|s-t\|_2.
\end{equation}
The function $L^*$ is continuous on $\ell_2^\downarrow$ in the $\ell_2$ norm.
If $s-t\in\ell_1$, then
\begin{equation}\label{eq:l1-infinite}
|L^*(s)-L^*(t)|\le\|s-t\|_1.
\end{equation}
\end{lemma}

\begin{proof}
Proposition~\ref{prop:strategy} expresses $L_n^*(s)$ as a finite minimum of
maxima of subset sums of $s_0,\ldots,s_{n-1}$. Each subset sum is $1$-Lipschitz in the $\ell_1$ norm;
finite maxima and minima preserve a common Lipschitz constant. This proves
\eqref{eq:finite-lipschitz}.

For $u\in\ell_2^\downarrow$ and $n\ge1$, write
$u^{[n]}=(u_0,\ldots,u_{n-1},0,0,\ldots)$ for the truncation of $u$, so that
$L_n^*(u)=L^*(u^{[n]})$ and, by~\eqref{eq:tail-estimate},
$|L^*(u)-L_n^*(u)|\le\|u-u^{[n]}\|_2$.
Fix $s\in\ell_2^\downarrow$. For any $t$ in that cone and any $n$,
\begin{align*}
|L^*(s)-L^*(t)|
&\le\sqrt n\|s-t\|_2
+\|s-s^{[n]}\|_2+\|t-t^{[n]}\|_2\\
&\le(\sqrt n+1)\|s-t\|_2
+2\|s-s^{[n]}\|_2.
\end{align*}
First choose $n$ to make the final term small, and then choose $t$ close to
$s$. This proves continuity. Finally, \eqref{eq:finite-lipschitz} gives
$|L_n^*(s)-L_n^*(t)|\le\|s-t\|_1$ for every $n$. Letting $n\to\infty$ and
using~\eqref{eq:truncation-limit} proves~\eqref{eq:l1-infinite}.
\end{proof}

The arguments below concern packings of actual squares, so they assume
that the relevant sides are positive. The next lemma shows that lower
bounds and formulas proved under this assumption hold in general.
\begin{lemma}[Reduction to positive sides]\label{lem:positive}
Let $n\ge1$, and let $f$ be a continuous real-valued function of
$(s_0,\ldots,s_{n-1})$.
\begin{enumerate}[label=\textup{(\roman*)}]
\item If $L_n^*(s)\ge f(s)$ for every $s\in\ell_2^\downarrow$ with
$s_{n-1}>0$, then $L^*(s)\ge L_n^*(s)\ge f(s)$ for all
$s\in\ell_2^\downarrow$.
\item If $L_n^*(s)=f(s)$ for every $s\in\ell_2^\downarrow$ with
$s_{n-1}>0$, then $L_n^*(s)=f(s)$ for all $s\in\ell_2^\downarrow$.
\end{enumerate}
\end{lemma}
\begin{proof}
Fix $s\in\ell_2^\downarrow$. For $\varepsilon>0$, let $s^\varepsilon$
be obtained from $s$ by adding $\varepsilon$ to each of
$s_0,\ldots,s_{n-1}$. Then $s^\varepsilon\in\ell_2^\downarrow$,
$s^\varepsilon_{n-1}>0$, and $s^\varepsilon\to s$ in $\ell_2$ as
$\varepsilon\to0$. By continuity of $L_n^*$ and $f$, letting
$\varepsilon\to0$ in the inequality or equality at $s^\varepsilon$
gives it at $s$. In case~(i), $L^*(s)\ge L_n^*(s)$ by monotonicity.
\end{proof}

\section{Packing up to six squares}\label{sec:small}

\subsection{Elementary chain estimates}
Fix a packing of squares $Q_i$ of sides $s_i$ in $[0,L]^2$. For a set of
indices $I$, call $\sum_{i\in I}s_i$ its \emph{side sum}; the same term
applies to the corresponding collection $\{Q_i:i\in I\}$.
We say that $Q_i$ lies \emph{left of} $Q_j$ when $x_i+s_i\le x_j$, and
\emph{below} $Q_j$ when $y_i+s_i\le y_j$. Two squares are \emph{$x$-disjoint}
when one lies left of the other, that is, when the interiors of their
horizontal projections are disjoint, and \emph{$x$-overlapping} otherwise.
Define $y$-disjointness and $y$-overlap using vertical projections.
By~\eqref{eq:sep}, every pair of squares in a packing is $x$-disjoint or
$y$-disjoint (possibly both). A \emph{horizontal chain} is a collection of
pairwise $x$-disjoint squares, and a \emph{vertical chain} is a collection of
pairwise $y$-disjoint squares. For positive side lengths, the nonempty
horizontal and vertical chains of a packing are exactly the chains of the strategy of the
packing defined in Section~\ref{sec:strategies}.

\begin{lemma}[Chain inequality]\label{lem:chain}
In a packing of side $L$, every horizontal or vertical chain has side sum
at most $L$. In particular, $L^*(s)\ge s_0+s_1$ for every
$s\in\ell_2^\downarrow$.
\end{lemma}
\begin{proof}
The projection intervals of a chain have disjoint interiors in $[0,L]$.
Hence the sum of their lengths is at most $L$. Any two positive squares form
a chain on at least one axis. If $s_1=0$, then
$s_0+s_1=s_0\le L$ by~\eqref{eq:contain}.
\end{proof}

\subsection{Four and five squares}
To check explicit coordinates, we use the elementary coefficient test for
nonincreasing side lengths. Set $s_n=0$. Then
\begin{equation}\label{eq:prefix-test}
\sum_{i=0}^{n-1}c_i s_i
=\sum_{k=0}^{n-1}\left(\sum_{i=0}^k c_i\right)(s_k-s_{k+1}).
\end{equation}
Thus, a linear form is nonnegative on every decreasing nonnegative tuple
if and only if all its coefficient prefix sums are nonnegative. Necessity
follows by taking the first $k+1$ sides equal to one and the rest zero. More generally, each coordinate containment or separation
inequality below has the form
$aL+\sum_{i=0}^{n-1}c_i s_i\ge0$ with $a\ge0$.
When $L$ is bounded below by the listed chain sums, we verify the inequality by replacing $L$ with one of those lower bounds and applying the coefficient test.

\begin{proposition}\label{prop:small-formulas}
For every $s\in\ell_2^\downarrow$,
\begin{equation}\label{eq:L4}
L_4^*(s)=s_0+s_1,
\end{equation}
\begin{equation}\label{eq:L5}
L_5^*(s)=\max(s_0+s_1,s_2+s_3+s_4).
\end{equation}
In particular, $L_1^*(s)=s_0$ and $L_2^*(s)=L_3^*(s)=s_0+s_1$.
\end{proposition}
\begin{proof}
By Lemma~\ref{lem:positive}, it suffices to prove~\eqref{eq:L4} when
$s_3>0$ and~\eqref{eq:L5} when $s_4>0$. Lemma~\ref{lem:chain} gives
$L_4^*(s)\ge s_0+s_1$. For $L\ge s_0+s_1$, take lower left corners $(0,0)$,
$(s_0,0)$, $(0,L-s_2)$, and $(L-s_3,L-s_3)$. Squares $Q_2,Q_3$ lie above
$Q_0,Q_1$, respectively, $Q_2$ lies left of $Q_1$ and $Q_3$, and $Q_0$
lies left of $Q_1$ and $Q_3$. These relations and containment in
$[0,L]^2$ follow from $L\ge s_0+s_1$ by the coefficient
test~\eqref{eq:prefix-test}. Thus, $L_4^*(s)\le s_0+s_1$. Since
$L_n^*(s)=L_4^*(s_0,\ldots,s_{n-1},0,0,\ldots)$ for $n\le3$, the formulas
for $L_1^*,L_2^*,L_3^*$ are special cases of~\eqref{eq:L4}.

For five squares and $L\ge\max(s_0+s_1,s_2+s_3+s_4)$, take lower left
corners $(0,0)$, $(L-s_1,0)$, $(0,L-s_2)$, $(s_2,L-s_3)$, and
$(L-s_4,L-s_4)$ as in the second panel of Figure~\ref{fig:small}. Squares
$Q_2,Q_3,Q_4$ form a horizontal chain along the top edge and lie above
$Q_0$ and $Q_1$, while $Q_0$ lies left of $Q_1$. These relations and
containment in $[0,L]^2$ follow from the two lower bounds for $L$ by the
coefficient test. Thus, $L_5^*(s)\le\max(s_0+s_1,s_2+s_3+s_4)$.
For the reverse inequality, it suffices to show that a packing of five
squares contains a horizontal or vertical chain of three squares, since
its side sum is at least $s_2+s_3+s_4$. Suppose no
three of the squares are pairwise $x$-disjoint. Let $A$ be the set of squares
with no square to their left, and let $B$ consist of the remaining squares.
Two squares in $A$ are not $x$-disjoint, since otherwise one would lie left
of the other. Two squares $Q_i,Q_j\in B$ are not $x$-disjoint either: if
$Q_i$ lay left of $Q_j$, then a square $Q_k$ left of $Q_i$ would make
$Q_k,Q_i,Q_j$ pairwise $x$-disjoint. Thus, the squares in each of $A$ and $B$
are pairwise $x$-overlapping, hence pairwise $y$-disjoint. One of $A,B$
contains at least three of the five squares, and these form a vertical
chain.
\end{proof}

\subsection{The six-square construction}
Our coordinate formulas will satisfy a fixed choice of horizontal or
vertical separation for every pair throughout the parameter range. The next
lemma records how these uniform separations determine a strategy and its
cost.

\begin{lemma}[Uniform coordinates with separations define a strategy]
\label{lem:coords}Let $g$ and $x_i,y_i$, $0\le i<n$, be real-valued functions of
$(s_0,\ldots,s_{n-1})$, with $g\ge0$.
Suppose that the coordinates $(x_i(s),y_i(s))$ pack
$s_0,\ldots,s_{n-1}$ in $[0,g(s)]^2$ for every
$s\in\ell_2^\downarrow$, with each pair satisfying a fixed
separation inequality from~\eqref{eq:sep} for all $s$.
Define $i\prec_{\mathrm h}j$ if $x_i(s)+s_i\le x_j(s)$ for all $s$,
and $i\prec_{\mathrm v}j$ if $y_i(s)+s_i\le y_j(s)$ for all $s$.
Then $\sigma=(\prec_{\mathrm h},\prec_{\mathrm v})$ is a strategy
with $\ell_\sigma\le g$. If $g$ is the maximum of the side sums
of a fixed collection of chains of $\sigma$, then $\ell_\sigma=g$.
\end{lemma}

\begin{proof}
The relations are transitive, and evaluating at positive side
vectors shows that they are irreflexive. The hypothesis on pairs
makes $\sigma$ a strategy, and the given packings follow $\sigma$.
Thus, Lemma~\ref{lem:strategy-side} gives $\ell_\sigma\le g$.
The final assertion follows from the definition of $\ell_\sigma$.
\end{proof}

We now give a single strategy for six squares that works for every
nonincreasing side vector.
\begin{proposition}[Six-square construction]\label{prop:six-upper}
For every $s\in\ell_2^\downarrow$,
\[
L_6^*(s)\le\max(s_0+s_1,\ s_2+s_3+s_4,\ s_1+s_4+s_5).
\]
More precisely, for $L$ at least the right-hand side, the coordinates
\begin{equation}\label{eq:six-coordinates}
\begin{array}{c|ccc}
i&0&1&2\\\hline
(x_i,y_i)&(0,0)&(L-s_1,0)&(0,L-s_2)\\[2pt]
i&3&4&5\\\hline
(x_i,y_i)&(s_2,L-s_3)&(L-s_4,L-s_4)&(L-s_5,s_1)
\end{array}
\end{equation}
pack $Q_0,\ldots,Q_5$ in $[0,L]^2$. Setting $L$ equal to the displayed
maximum defines a strategy whose cost is that maximum.
\end{proposition}
\begin{proof}
Let $g(s)$ be the displayed maximum. For every $L\ge g(s)$, the
coefficient test~\eqref{eq:prefix-test} verifies that the squares lie
in $[0,L]^2$ and that each pair satisfies a fixed separation inequality. The
collections $\{Q_0,Q_1\}$ and $\{Q_2,Q_3,Q_4\}$ are horizontal chains, while
$Q_1,Q_5,Q_4$ form a vertical chain in that order. Setting $L=g(s)$ and
applying Lemma~\ref{lem:coords} therefore gives a strategy with cost $g(s)$.
Proposition~\ref{prop:strategy} gives the upper bound for
$L_6^*$.
\end{proof}
Figure~\ref{fig:small} illustrates the constructions.

\subsection{Four large squares}
\begin{figure}[!htbp]\centering
\begin{minipage}[t]{0.32\textwidth}\centering
\begin{tikzpicture}[x=0.17cm,y=0.17cm]
\useasboundingbox (0,-0.4) rectangle (20.4,20);
\filldraw[fill=black!7,draw=black,line width=.35pt] (3,3) rectangle (12,12);
\node[font=\scriptsize,inner sep=0pt] at (7.5,7.5) {$0$};
\filldraw[fill=black!7,draw=black,line width=.35pt] (12,0) rectangle (20,8);
\node[font=\scriptsize,inner sep=0pt] at (16,4) {$1$};
\filldraw[fill=black!7,draw=black,line width=.35pt] (0,12) rectangle (7,19);
\node[font=\scriptsize,inner sep=0pt] at (3.5,15.5) {$2$};
\filldraw[fill=black!7,draw=black,line width=.35pt] (13,11) rectangle (19,17);
\node[font=\scriptsize,inner sep=0pt] at (16,14) {$3$};
\draw[line width=.7pt] (0,0) rectangle (20,20);
\end{tikzpicture}\par
\smallskip{\footnotesize (a) windowpane, $L=20$\\ $s=(9,8,7,6)$}
\end{minipage}\hfill
\begin{minipage}[t]{0.32\textwidth}\centering
\begin{tikzpicture}[x=0.17cm,y=0.17cm]
\useasboundingbox (0,-0.4) rectangle (20.4,20);
\draw[densely dashed,line width=.35pt,black!55] (3,3) rectangle (12,12);
\filldraw[fill=black!7,draw=black,line width=.35pt] (0,0) rectangle (9,9);
\node[font=\scriptsize,inner sep=0pt] at (4.5,4.5) {$0$};
\filldraw[fill=black!7,draw=black,line width=.35pt] (12,0) rectangle (20,8);
\node[font=\scriptsize,inner sep=0pt] at (16,4) {$1$};
\filldraw[fill=black!7,draw=black,line width=.35pt] (0,12) rectangle (7,19);
\node[font=\scriptsize,inner sep=0pt] at (3.5,15.5) {$2$};
\filldraw[fill=black!7,draw=black,line width=.35pt] (13,11) rectangle (19,17);
\node[font=\scriptsize,inner sep=0pt] at (16,14) {$3$};
\draw[-{Stealth[length=3pt]},line width=.45pt] (12.3,12.3) -- (9.4,9.4);
\draw[line width=.7pt] (0,0) rectangle (20,20);
\end{tikzpicture}\par
\smallskip{\footnotesize (b) cornering, $L=20$\\ $s=(9,8,7,6)$}
\end{minipage}\hfill
\begin{minipage}[t]{0.32\textwidth}\centering
\begin{tikzpicture}[x=0.0773cm,y=0.0773cm]
\useasboundingbox (-3.96,-4.84) rectangle (44.88,44);
\draw[densely dashed,line width=.35pt,black!55] (24,0) rectangle (44,20);
\filldraw[fill=black!7,draw=black,line width=.35pt] (0,0) rectangle (21,21);
\node[font=\scriptsize,inner sep=0pt] at (10.5,10.5) {$0$};
\filldraw[fill=black!7,draw=black,line width=.35pt] (24,7) rectangle (44,27);
\node[font=\scriptsize,inner sep=0pt] at (34,17) {$1$};
\filldraw[fill=black!7,draw=black,line width=.35pt] (0,26) rectangle (18,44);
\node[font=\scriptsize,inner sep=0pt] at (9,35) {$p{=}2$};
\filldraw[fill=black!7,draw=black,line width=.35pt] (18,27) rectangle (35,44);
\node[font=\scriptsize,inner sep=0pt] at (26.5,35.5) {$c{=}3$};
\filldraw[fill=black!7,draw=black,line width=.35pt] (35,35) rectangle (44,44);
\node[font=\scriptsize,inner sep=0pt] at (39.5,39.5) {$4$};
\filldraw[fill=black!7,draw=black,line width=.35pt] (36,27) rectangle (44,35);
\node[font=\scriptsize,inner sep=0pt] at (40,31) {$5$};
\draw[densely dashed,line width=.3pt] (21,0) -- (21,44);
\node[font=\tiny,inner sep=1pt,below] at (21,0) {$s_0$};
\draw[densely dashed,line width=.3pt] (0,20) -- (44,20);
\node[font=\tiny,inner sep=1pt,left] at (0,20) {$s_1$};
\draw[-{Stealth[length=3pt]},line width=.45pt] (34,11) -- (34,4);
\draw[line width=.7pt] (0,0) rectangle (44,44);
\end{tikzpicture}\par
\smallskip{\footnotesize (c) forced relations, $L=44$\\ $s=(21,20,18,17,9,8)$}
\end{minipage}
\caption{(a) A windowpane in a square of side $20<s_1+s_2+s_3$: the columns are
$\{Q_0,Q_2\}$ and $\{Q_1,Q_3\}$, and the rows are $\{Q_0,Q_1\}$ and
$\{Q_2,Q_3\}$. (b) The cornering step of Lemma~\ref{lem:pane}: $Q_0$ has no
blocker to its left or below, and moves left and then down from the
dashed position to the lower left corner. (c) The conclusions of Lemma~\ref{lem:fixed-core} with $r=N=5$,
where every triple in the hypotheses has side sum at least
$s_1+s_3+s_5=45>L$. Here $p=2$ and $c=3$; the squares $Q_4,Q_5$ are
$x$-disjoint from $Q_p,Q_c$, $y$-overlap both, and lie above height $s_1$.
Square $Q_1$ can be lowered to the floor. The dashed lines are $x=s_0$
and $y=s_1$, and the dashed outline is the lowered position of $Q_1$.}
\label{fig:pane}
\end{figure}

The five-square lower bound needed only the existence of a three-square
chain. For six squares we need more information, and we first set up the
method used for all remaining lower bounds.

Suppose that we wish to prove a lower bound of the form
$L^*(s)\ge\Sigma$. We assume to the contrary that some packing fits in a square of
side $L<\Sigma$. We say that a set of indices $I$, or the corresponding collection
of squares $\{Q_i:i\in I\}$, \emph{clears} the threshold $\Sigma$
when $\sum_{i\in I}s_i\ge\Sigma$. By Lemma~\ref{lem:chain}, such a collection cannot be a
horizontal or vertical chain. The proofs below use these exclusions to force overlaps and then analyze the resulting geometry to get a contradiction.

For example, suppose $Q_a,Q_b$ are $y$-disjoint and $\{a,b,e\}$ clears $\Sigma$.
Then $Q_e$ must $y$-overlap $Q_a$ or $Q_b$, since otherwise the three
squares would form a vertical chain with side sum exceeding $L$.
Whenever two squares $y$-overlap, they must be $x$-disjoint. We call the
forced $y$-overlap with one of $Q_a,Q_b$ a \emph{vertical pin} of $Q_e$
to the pair $Q_a,Q_b$. A pin to a longer chain, or a horizontal pin,
means the analogous necessary overlap.

We will sometimes translate a square to a wall while keeping the other
squares fixed. For instance, a \emph{blocker to the right} of $Q_i$
is a square to its right whose vertical projection overlaps that of
$Q_i$: moving $Q_i$ right would eventually meet it. If there is no
such blocker, $Q_i$ can be moved to the right wall without disturbing
the packing. Blockers in other directions are defined analogously.
Reflections and interchanges of the axes move the entire packing.
Because clearing depends only on the side lengths, an exclusion based
on a clearing set remains valid during any motion that preserves a packing.

If every triple among the four largest squares clears $\Sigma$, their projections
must have a constrained arrangement, which we call a \emph{windowpane}
(Figure~\ref{fig:pane}(a)). Its ``columns'' will be $x$-overlapping pairs and
its ``rows'' $y$-overlapping pairs. With an additional clearing hypothesis,
this structure also lets us move the largest square into a corner of the
full packing (Figure~\ref{fig:pane}(b)). Both facts will be
used again for seven squares.
\begin{lemma}[Windowpane and cornering]\label{lem:pane}
Fix a packing of side $L<\Sigma$, and suppose every triple among $Q_0,Q_1,Q_2,Q_3$ clears $\Sigma$. The four squares split
into two $x$-overlapping pairs, called columns, and into two
$y$-overlapping pairs, called rows. No square $x$-overlaps
both members of the other column, or $y$-overlaps both members of
the other row.

If, in addition, every triple $\{0,b,j\}$ with
$b\in\{1,2,3\}$ and $j\notin\{0,b\}$ in the full packing clears $\Sigma$, then
$Q_0$ can be translated to a corner in that full packing.
\end{lemma}
\begin{proof}
Make a graph with vertices $0,1,2,3$, joining two indices when their
squares are $x$-disjoint. The graph has no triangle, since three pairwise $x$-disjoint squares
would form a clearing horizontal chain. Since it has
only four vertices, it has no odd cycle and is bipartite. Each of its
two parts consists of pairwise $x$-overlapping squares. Such squares are
pairwise $y$-disjoint, and a part cannot contain three indices, since a
vertical triple would clear $\Sigma$. Hence the parts both have size two, giving
the two columns. Repeating the argument with the coordinate axes
interchanged gives the two rows. A square that $x$-overlapped both members
of the other column would form a forbidden vertical triple with them.
The row assertion follows by symmetry.

For the cornering statement, consider blockers on opposite sides of
$Q_0$ in the horizontal direction. They are $x$-disjoint from $Q_0$ and
from each other. Together with $Q_0$ they form a horizontal chain.
If either blocker is among $Q_1,Q_2,Q_3$, this triple clears $\Sigma$. If both
are smaller squares, there is a square $Q_b$, $b\in\{1,2,3\}$,
$x$-disjoint from $Q_0$: otherwise $Q_0$ and both squares of the other column would
form a forbidden vertical triple. Square $Q_b$ is on one side of
$Q_0$ and, together with the blocker on the opposite side, forms a
clearing horizontal triple with $Q_0$. Thus, there cannot be blockers
on both sides. Translate $Q_0$ to a vertical wall, and repeat the
argument vertically to reach a corner; reflect the packing if needed.
All the side-sum exclusions remain valid during these motions.
\end{proof}

The next lemma concerns the $r+1$ largest squares $Q_0,\ldots,Q_r$ of a
packing of $Q_0,\ldots,Q_N$, and its hypotheses involve only indices in
$\{0,\ldots,r\}$. In the proof we remove $Q_{r+1},\ldots,Q_N$, lower $Q_1$
in the resulting packing of $Q_0,\ldots,Q_r$, derive relations among squares
whose positions have not changed, and then use those relations in the
original packing. The auxiliary translation need not be possible when
$Q_{r+1},\ldots,Q_N$ are restored.
Figure~\ref{fig:pane}(c) illustrates the conclusions.
\begin{lemma}[Relations forced by four large squares]\label{lem:fixed-core}
Suppose $Q_0=[0,s_0]^2$ in a packing of $Q_0,\ldots,Q_N$ of side $L<\Sigma$, where $N\ge r\ge4$.
Assume that the following triples of indices in $\{0,\ldots,r\}$ clear $\Sigma$:
\begin{enumerate}[label=\textup{(\roman*)}]
\item every triple among $0,1,2,3$;
\item $\{0,1,j\}$ for $2\le j\le r$;
\item $\{0,b,m\}$ and $\{1,b,m\}$ for $b\in\{2,3\}$ and $4\le m\le r$.
\end{enumerate}
After possibly interchanging axes, $x_1\ge s_0$ and $y_1<s_0$.
There are distinct $p,c\in\{2,3\}$ such that $p$ $x$-overlaps $0$,
$c$ $x$-overlaps $1$, and $c$ lies above $1$. Moreover, in the
original full packing:
\begin{enumerate}[label=\textup{(\alph*)}]
\item $p,c$ are $x$-disjoint, with $p$ left of $c$;
\item each $m\in\{4,\ldots,r\}$ is $x$-disjoint from both $p,c$ and $y$-overlaps
      both of them;
\item $y_m\ge s_1$ for every such $m$.
\end{enumerate}
If $r=N$, square $Q_1$ can be moved to the floor in the full packing.
\end{lemma}
\begin{proof}
Apply the windowpane statement and choose axes so that $x_1\ge s_0$.
Suppose first that $y_1\ge s_0$. Let $p$ and $q$ be the column and row
partners of $0$; they are $2,3$ in some order. Then $p$ lies above $0$,
$q$ lies right of $0$, and $1$ lies right of $p$ and above $q$. These
relations follow from the windowpane, since $1,p$ lie in a row and $1,q$ in
a column, together with $x_p<s_0\le x_1$ and $y_q<s_0\le y_1$.

Set $m=4$. The vertical pin of $Q_m$ to $Q_0,Q_1$ forces it to $y$-overlap one of them. If it
overlaps $0$, it lies right of $0$, and the vertical pin to $Q_1,Q_q$ forces
an overlap with $1$ or $q$. It is $x$-disjoint from that square, and both
are right of $0$, giving a clearing horizontal triple. Otherwise $m$ is
$y$-disjoint from $0$ and $y$-overlaps $1$. It is $x$-disjoint from $1$;
the horizontal pin to $Q_p,Q_1$ makes it $x$-overlap $Q_p$. The vertical
chain $Q_0,Q_p,Q_m$ now clears $\Sigma$, a contradiction. Hence $y_1<s_0$.

Let $c$ be the column partner of $1$ and $p$ the column partner of $0$.
Suppose $c$ lies below $1$. Then $y_c+s_c\le y_1<s_0$. Hence $c$ lies right
of $0$, whereas $p$ lies above $0$. Square $p$ cannot lie below $1$, and
cannot lie above it because $Q_c,Q_1,Q_p$ would be a vertical chain. Thus, $p$
$y$-overlaps $1$ and is $x$-disjoint from it.

Again take $m=4$. The vertical pin of $Q_m$ to $Q_c,Q_1$ forces a $y$-overlap. If it overlaps
$c$ but not $1$, it lies below $1$ and right of $0$, and is $x$-disjoint from
$c$, giving the clearing horizontal chain $Q_0,Q_c,Q_m$. Therefore, $Q_m$ $y$-overlaps $Q_1$ and is $x$-disjoint from it. Hence the horizontal pin to \(Q_0,Q_1\) forces \(Q_m\) to \(x\)-overlap \(Q_0\). Thus, $Q_m$ is
$y$-disjoint from $Q_0=[0,s_0]^2$ and lies above it. If $m,p$ are $x$-disjoint, the horizontal chain $Q_1,Q_p,Q_m$ clears $\Sigma$;
otherwise the vertical chain $Q_0,Q_p,Q_m$ clears $\Sigma$. Hence $c$ lies above $1$.

Now remove $Q_{r+1},\ldots,Q_N$. In the packing of $Q_0,\ldots,Q_r$, square $1$ can be
lowered to the floor with all the other squares fixed. A downward blocker
other than $0$ would complete a clearing vertical triple with $1,c$,
while $0$ is $x$-disjoint from $1$. In the auxiliary placement, $Q_0$ and the
lowered $Q_1$ both lie on the floor. Each of $Q_p,Q_4,\ldots,Q_r$ is pinned
horizontally to $Q_0,Q_1$. It therefore $x$-overlaps one of them, is
$y$-disjoint from it, and lies above it. Since $s_0\ge s_1$, each of these squares lies above
height $s_1$. The vertical pins to $Q_1,Q_c$ make these squares $y$-overlap $Q_c$,
hence be $x$-disjoint from it. The pins to $Q_1,Q_p$ likewise make
each $m\ge4$ $y$-overlap $p$, hence be $x$-disjoint from it.

All coordinates occurring in (a)--(c) have remained fixed; only the
vertical coordinate of $1$ changed. Thus, these assertions hold in the
original full packing, regardless of $Q_{r+1},\ldots,Q_N$. Finally, $p$
lies left of $c$: the reverse order would give
$x_c+s_c\le x_p<s_0\le x_1$, contradicting the $x$-overlap of $c,1$.
If $r=N$, the auxiliary motion is a motion of the full packing.
\end{proof}

\subsection{A column bound and the six-square formula}

\begin{proposition}[Column bound]\label{prop:column-bound}
For every $n\ge6$ and $s\in\ell_2^\downarrow$,
\begin{equation}\label{eq:column-bound}
L^*(s)\ge L_n^*(s)\ge\min\left(s_1+s_3+s_{n-1},\ P_n(s)\right),
\qquad\text{where }P_n(s)=s_1+\sum_{i=4}^{n-1}s_i.
\end{equation}
\end{proposition}
\begin{proof}
By Lemma~\ref{lem:positive}, it suffices to prove the inequality for
$L_n^*$ when $s_{n-1}>0$. Set
\[
D=s_1+s_3+s_{n-1}\quad\text{and}\quad Z=s_2+s_3+s_{n-2}+s_{n-1},
\]
and suppose that $L<\Sigma:=\min(D,P_n,Z)$. Every triple of indices among $0,1,2,3$ clears $\Sigma$; so do $\{0,b,j\}$ with $b\in\{1,2,3\}$ and any
other index $j$, and $\{1,b,m\}$ with $b\in\{2,3\}$ and $m\ge4$.
Each of these triples has side sum at least $D$, because $s$ is nonincreasing. By Lemma~\ref{lem:pane}, we may place $Q_0$ in the lower left corner. Lemma~\ref{lem:fixed-core}, with $r=N=n-1$, then allows us,
after possibly interchanging axes, to lower $Q_1$ to the floor.
It also shows that $Q_2,Q_3$ are $x$-disjoint from one another
and from every square with index at least $4$, all of which
lie above $Q_1$.

If two of the squares $Q_4,\ldots,Q_{n-1}$ are $x$-disjoint, they and
$Q_2,Q_3$ form a horizontal chain of sum at least $Z$. Otherwise these
squares are pairwise $x$-overlapping and form a vertical chain. With $Q_1$
below all of them, this chain has sum $P_n$. Either chain contradicts
$L<\Sigma$. Thus, $L\ge\min(D,P_n,Z)$ in every packing.

It remains to show that $L^*_n \geq \min(D,P_n)$. If $\min(D,P_n)\le s_0+s_1$,
use Lemma~\ref{lem:chain}. Otherwise $D>s_0+s_1$, which gives
$s_3+s_{n-1}>s_0\ge s_1$ and $s_2+s_{n-2}\ge s_3+s_{n-1}>s_1$. It follows
that $Z>D$, proving \eqref{eq:column-bound}.
\end{proof}

\begin{theorem}\label{thm:six}
For every $s\in\ell_2^\downarrow$,
\begin{equation}\label{eq:L6}
L_6^*(s)=\max(s_0+s_1,\ s_2+s_3+s_4,\ s_1+s_4+s_5).
\end{equation}
\end{theorem}
\begin{proof}
Proposition~\ref{prop:six-upper} proves the upper bound.
Lemma~\ref{lem:chain} gives $L_6^*\ge s_0+s_1$, and
$L_6^*\ge L_5^*\ge s_2+s_3+s_4$ by~\eqref{eq:L5}. Proposition~\ref{prop:column-bound} with $n=6$ gives the remaining lower bound
$L_6^*\ge\min(s_1+s_3+s_5,\ s_1+s_4+s_5)=s_1+s_4+s_5$,
because $s_3\ge s_4$. This proves~\eqref{eq:L6}.
\end{proof}

\section{Sharp bounds on the number of sides in a sum}\label{sec:longest}
Proposition~\ref{prop:strategy} guarantees that every value $L_n^*(s)$ is some
subset sum. We now determine how many sides that sum may need. For
$s\in\ell_2^\downarrow$, call $I\subseteq\{0,\ldots,n-1\}$ a
\emph{representing set} if $L_n^*(s)=\sum_{i\in I}s_i$, and put
$q_n(s)=\min\{|I|:I\text{ represents }s\}$.
We prove $\max_{s\in\ell_2^\downarrow}q_n(s)=n-3$ for $n\ge6$.

\subsection{Representing sums with at most \texorpdfstring{$n-3$}{n-3} sides}\label{sec:short-sums}
We begin with two placements that work for any number of squares.
\begin{proposition}[Two general constructions]\label{prop:two-placements}
For $n\ge6$ and $P_n$ as in Proposition~\ref{prop:column-bound}, one has
\begin{equation}\label{eq:column}
L_n^*(s)\le\max(s_0+s_1,\ s_2+s_3+s_4,\ P_n).
\end{equation}
Also,
\[
L_n^*(s)\le\max\bigg(s_0+s_1,\ s_0+s_2+s_3,\
\sum_{\substack{2\le i<n\\i\ \mathrm{even}}}s_i\bigg).
\]
In particular, if $s_2+s_3\le s_1$, then
\begin{equation}\label{eq:shelves-small}
L_n^*(s)\le\max\bigg(s_0+s_1,\
\sum_{\substack{2\le i<n\\i\ \mathrm{even}}}s_i\bigg).
\end{equation}
\end{proposition}
\begin{proof}
For the first construction, set $c=\max(s_0,s_2+s_3)$ and place
$Q_0,\ldots,Q_3$ with respective lower left corners
\[
(0,0),\quad (s_0,0),\quad (0,s_0),\quad (s_2,s_0).
\]
For $4\le i\le n-1$, put
$(x_i,y_i)=(c,\ s_1+\sum_{j=4}^{i-1}s_j)$.
The column squares are mutually $y$-disjoint, lie above $Q_1$, and lie right
of $Q_0,Q_2,Q_3$. The shelf squares $Q_2,Q_3$ are $x$-disjoint and lie above
the two bottom squares. The horizontal extent is
$\max(s_0+s_1,c+s_4)=\max(s_0+s_1,s_2+s_3+s_4)$, since $s_4\le s_1$.
The vertical extent is $\max(s_0+s_2,P_n)\le\max(s_0+s_1,P_n)$.

For the second construction, retain $Q_0,Q_1$ at the bottom. Place the
even-indexed squares $Q_2,Q_4,\ldots$ consecutively from left to right
on the shelf at height $s_0$, and the odd-indexed squares
$Q_3,Q_5,\ldots$ on the shelf at height $s_0+s_2$.
Their heights are at most $s_2$ and $s_3$, respectively. By termwise
comparison, the odd-indexed shelf is no wider than the even-indexed one.
The bounds follow from the horizontal and vertical extents.
\end{proof}

We next show that representing sets with at least $n-2$ elements can be
avoided.

\begin{theorem}[Representing sums with at most $n-3$ sides]\label{thm:long-upper}
Let $n\ge6$ and $s\in\ell_2^\downarrow$ with $s_{n-1}>0$. If a
representing set $I$ has $|I|\ge n-2$, then $L_n^*(s)=s_0+s_1$.
For every $s\in\ell_2^\downarrow$, including vectors with zero entries,
$q_n(s)\le n-3$.
\end{theorem}
\begin{proof}
Let $J=I^c$. Then $|J|\le2$, so $\sum_{j\in J}s_j\le s_0+s_1$ and
\begin{equation}\label{eq:tail-representing}
\sum_{i\in I}s_i\ge\sum_{i=2}^{n-1}s_i.
\end{equation}
The chain inequality gives $\sum_{i\in I}s_i=L_n^*\ge s_0+s_1$. Suppose it is strictly
larger; we obtain a packing of side strictly less than $\sum_{i\in I}s_i$.

If $s_1<s_2+s_3$, use Proposition~\ref{prop:two-placements}. Its first term is
smaller than $\sum_{i\in I}s_i$ by assumption. Its second is smaller by
\eqref{eq:tail-representing}, since the tail beginning with $s_5$ is
nonempty and positive. For the third, both sums omit entries from the sum of all $n$ sides, whence
\[
\sum_{i\in I}s_i-P_n=(s_0+s_2+s_3)-\sum_{j\in J}s_j\ge s_2+s_3-s_1>0.
\]

If $s_2+s_3\le s_1$, use~\eqref{eq:shelves-small}. The first term is smaller
than $\sum_{i\in I}s_i$. The even-indexed sum beginning with $s_2$ is strictly smaller
than $\sum_{i=2}^{n-1}s_i\le \sum_{i\in I}s_i$, because the latter also includes
$s_3>0$. This is again a contradiction.

Thus, for positive side lengths, every representing set of cardinality at least \(n-2\) can be replaced by \(\{0,1\}\), proving \(q_n(s)\le n-3\). If some entries vanish, apply this result to the positive prefix when it has at least six entries, and otherwise use the formulas for at most five squares.
\end{proof}
The two constructions serve different parameter ranges. For example,
with $n=10$ and sides $(5,5,2,\ldots,2)$, the column construction gives an
upper bound of $17$,
whereas the two-shelf construction gives the exact value $10$.

\subsection{Sharpness and consequences}\label{sec:sharpness}
We now show that the bound $n-3$ is attained on an open set of side
vectors.
\begin{figure}[b]\centering
\begin{minipage}[t]{0.5\textwidth}\centering
\begin{tikzpicture}[x=0.060724cm,y=0.060724cm]
\useasboundingbox (0,-9.88) rectangle (85.88,76);
\filldraw[fill=black!7,draw=black,line width=.35pt] (0,0) rectangle (36,36);
\node[font=\scriptsize,inner sep=0pt] at (18,18) {$0$};
\filldraw[fill=black!7,draw=black,line width=.35pt] (36,0) rectangle (72,36);
\node[font=\scriptsize,inner sep=0pt] at (54,18) {$1$};
\filldraw[fill=black!7,draw=black,line width=.35pt] (0,36) rectangle (33,69);
\node[font=\scriptsize,inner sep=0pt] at (16.5,52.5) {$2$};
\filldraw[fill=black!7,draw=black,line width=.35pt] (33,36) rectangle (66,69);
\node[font=\scriptsize,inner sep=0pt] at (49.5,52.5) {$3$};
\filldraw[fill=black!7,draw=black,line width=.35pt] (68,36) rectangle (76,44);
\node[font=\scriptsize,inner sep=0pt] at (72,40) {$4$};
\filldraw[fill=black!7,draw=black,line width=.35pt] (68,44) rectangle (76,52);
\node[font=\scriptsize,inner sep=0pt] at (72,48) {$5$};
\filldraw[fill=black!7,draw=black,line width=.35pt] (68,52) rectangle (76,60);
\node[font=\scriptsize,inner sep=0pt] at (72,56) {$6$};
\filldraw[fill=black!7,draw=black,line width=.35pt] (68,60) rectangle (76,68);
\node[font=\scriptsize,inner sep=0pt] at (72,64) {$7$};
\filldraw[fill=black!7,draw=black,line width=.35pt] (68,68) rectangle (76,76);
\node[font=\scriptsize,inner sep=0pt] at (72,72) {$8$};
\draw[line width=.7pt] (0,0) rectangle (76,76);
\filldraw[fill=vc!25,draw=vc!70!black,line width=.25pt] (79.04,0) rectangle (84.74,36);
\node[font=\tiny,inner sep=0pt] at (81.89,18) {$1$};
\filldraw[fill=vc!25,draw=vc!70!black,line width=.25pt] (79.04,36) rectangle (84.74,44);
\node[font=\tiny,inner sep=0pt] at (81.89,40) {$4$};
\filldraw[fill=vc!25,draw=vc!70!black,line width=.25pt] (79.04,44) rectangle (84.74,52);
\node[font=\tiny,inner sep=0pt] at (81.89,48) {$5$};
\filldraw[fill=vc!25,draw=vc!70!black,line width=.25pt] (79.04,52) rectangle (84.74,60);
\node[font=\tiny,inner sep=0pt] at (81.89,56) {$6$};
\filldraw[fill=vc!25,draw=vc!70!black,line width=.25pt] (79.04,60) rectangle (84.74,68);
\node[font=\tiny,inner sep=0pt] at (81.89,64) {$7$};
\filldraw[fill=vc!25,draw=vc!70!black,line width=.25pt] (79.04,68) rectangle (84.74,76);
\node[font=\tiny,inner sep=0pt] at (81.89,72) {$8$};
\end{tikzpicture}\par
\smallskip{\footnotesize $L_9^*=76$\\ $s=(36,36,33,33,8,8,8,8,8)$}
\end{minipage}
\caption{The member of the sharpness family \eqref{eq:sharp-family}
with $m=4$. The six-term vertical chain has side sum $76$,
whereas $s_0+s_1=72$, $s_2+s_3+s_4=74$, and $s_1+s_3+s_8=77$.
The lower bound proves that the column cost is optimal.}
\label{fig:long-column}
\end{figure}
\begin{theorem}[Sharpness]\label{thm:long-sharp}
For every $n\ge6$, there is a nonempty open set $U$ of strictly
decreasing positive $n$-tuples such that, for all $s\in\ell_2^\downarrow$
with $(s_0,\ldots,s_{n-1})\in U$,
\[
L_n^*(s)=P_n(s)=s_1+\sum_{i=4}^{n-1}s_i\quad\text{and}\quad q_n(s)=n-3.
\]
\end{theorem}
\begin{proof}
Let $n\ge6$, set $m=n-5$, and consider
\begin{equation}\label{eq:sharp-family}
s_0=s_1=8m+4,\qquad s_2=s_3=8m+1,\qquad
s_4=\cdots=s_{n-1}=8.
\end{equation}
There are $m+1$ sides equal to $8$. At this input, $P_n=16m+12$ and
\[
s_0+s_1=P_n-4,\qquad
s_2+s_3+s_4=P_n-2,\qquad
s_1+s_3+s_{n-1}=P_n+1.
\]
The column construction and the column bound of
Proposition~\ref{prop:column-bound} give $L_n^*=P_n$. Figure~\ref{fig:long-column} shows the member of this family with $m=4$. The three displayed comparisons remain strict in a neighborhood of this
vector. Intersecting that neighborhood with the cone of strictly decreasing
positive $n$-tuples gives a nonempty open set on which $L_n^*=P_n$.

For each subset $I$ of cardinality less than
$n-3$, the linear forms $P_n$ and $\sum_{i\in I}s_i$ have different
coefficient vectors. Their equality defines a proper hyperplane. Since
there are only finitely many such subsets, we may choose a smaller open
ball that avoids all of these hyperplanes. On that ball, $P_n$ has no
representation using fewer than $n-3$ sides.
\end{proof}
Finally, we deduce consequences for linear components, for
minimum-of-maxima formulas, and for Lipschitz constants. A \emph{linear
component} means a linear function of $(s_0,\ldots,s_{n-1})$
that agrees with $L_n^*(s)$ on a nonempty open set of strictly decreasing
positive $n$-tuples.

\begin{corollary}[Minimum-of-maxima formulas with at most $n-3$ sides per sum]\label{cor:short-formula}
For $n\ge6$, every linear component is a subset sum with at most $n-3$
entries. There is a finite minimum-of-maxima formula for $L_n^*(s)$ in
$s_0,\ldots,s_{n-1}$ whose constituent sums all have that cardinality bound.
No such representation using only shorter subset sums exists.
\end{corollary}
\begin{proof}
The finite strategy formula subdivides the cone of nonincreasing
nonnegative $n$-tuples into finitely many polyhedral regions on which
$L_n^*(s)$ is a subset sum. A linear component
agrees with one of these subset sums on an open subset of its region, and
hence is a subset sum. On a full-dimensional
region, Theorem~\ref{thm:long-upper} excludes a component of cardinality
at least $n-2$: it would equal $s_0+s_1$ on an open set, hence as a linear
form, contrary to its different coefficient vector. Boundary values follow
by continuity from the full-dimensional regions.

The lattice representation theorem \cite{Ovchinnikov02}, applied to \(-L_n^*\) on the closed convex cone of nonincreasing nonnegative \(n\)-tuples, expresses \(L_n^*\) as a finite minimum of maxima of its linear components. These are the short subset sums identified above. Finally, any finite
minimum-of-maxima expression selects one of its constituent values at every
input. Theorem~\ref{thm:long-sharp} therefore rules out using only shorter
subset sums.
\end{proof}

\begin{corollary}[Sharp Euclidean Lipschitz constants]\label{cor:sharp-lipschitz}
For every $n\ge6$,
\[
|L_n^*(s)-L_n^*(t)|\le\sqrt{n-3}\,\|s-t\|_2\quad\text{for all }s,t\in\ell_2^\downarrow,
\]
and the constant $\sqrt{n-3}$ cannot be decreased. In particular $L^*$ is not globally Lipschitz on $\ell_2^\downarrow$ in the
$\ell_2$ metric.
\end{corollary}
\begin{proof}
A sum of at most $n-3$ coordinates has gradient of Euclidean norm at most
$\sqrt{n-3}$. Since finite minima and maxima preserve a common Lipschitz constant,
Corollary~\ref{cor:short-formula} gives
\[
|L_n^*(s)-L_n^*(t)|\le\sqrt{n-3}\Bigl(\sum_{i<n}(s_i-t_i)^2\Bigr)^{1/2}\le\sqrt{n-3}\,\|s-t\|_2.
\]
Now, by Theorem~\ref{thm:long-sharp}, \(L_n^*=P_n\) on a nonempty open set of strictly decreasing positive \(n\)-tuples. Since \(\|\nabla P_n\|_2=\sqrt{n-3}\), the constant is sharp. On sequences obtained by extending \(n\)-tuples with zeros, \(L^*=L_n^*\). Thus, a global Lipschitz constant for \(L^*\) would have to be at least \(\sqrt{n-3}\) for every \(n\), which is impossible.
\end{proof}

Together, Theorem~\ref{thm:long-upper}, Theorem~\ref{thm:long-sharp}, and the
two corollaries prove Theorem~\ref{thm:intro-longest}.

\section{Seven squares}\label{sec:seven}
\subsection{Three strategies}
Put
\begin{gather}\label{eq:seven-sums}
W=s_3+s_4+s_5+s_6,\qquad H=s_1+s_2+s_5,\\
C_0=s_0+s_5+s_6,\qquad C_3=s_1+s_3+s_6,\qquad R=s_1+s_4+s_5+s_6.\nonumber
\end{gather}
and
\begin{equation}\label{eq:F7}
F_7=\max\bigl(L_6^*,\min(\max(W,H),\max(C_0,C_3),R)\bigr).
\end{equation}
We use the coordinates in Table~\ref{tab:coords67} to construct three
strategies. They are named
for the position of $Q_6$: against the top edge, in the interior, and
against the right edge.
\begin{proposition}[Three seven-square strategies]\label{prop:seven-upper}
Setting $L$ in each row of Table~\ref{tab:coords67} equal to the corresponding
right-hand side below defines strategies
$\sigma_{\st{top}},\sigma_{\st{center}},\sigma_{\st{right}}$ with costs
\[
\ell_{\st{top}}=\max(L_6^*,W,H),\quad
\ell_{\st{center}}=\max(L_6^*,C_0,C_3),\quad
\ell_{\st{right}}=\max(L_6^*,R).
\]
For each strategy and every $L$ at least its cost, the corresponding row of
Table~\ref{tab:coords67} gives a packing in $[0,L]^2$. Consequently, for
every $s\in\ell_2^\downarrow$,
\[
L_7^*(s)\le\min(\ell_{\st{top}},\ell_{\st{center}},\ell_{\st{right}})=F_7(s).
\]
\end{proposition}
\begin{proof}
For each row, let $g(s)$ be the corresponding right-hand side in the statement.
The coefficient test~\eqref{eq:prefix-test} verifies that the squares lie
in $[0,L]^2$ and that each pair satisfies a fixed separation inequality
whenever $L\ge g(s)$. Setting $L=g(s)$ and applying Lemma~\ref{lem:coords}
gives a strategy with cost at most $g(s)$.
For example, in the center strategy $Q_6$ lies right of $Q_0$
because $x_6=L-s_5-s_6\ge s_0$ when $L\ge C_0$, and below $Q_3$ because
$y_6+s_6=s_1+s_6\le L-s_3=y_3$ when $L\ge C_3$. The following chain sums give the reverse inequality:
\[
\begin{array}{c|l|l}
&\text{horizontal}&\text{vertical}\\\hline
\st{top}&s_0+s_1,s_2+s_3+s_4,s_3+s_4+s_5+s_6&s_1+s_2+s_5\\
\st{center}&s_0+s_1,s_2+s_3+s_4,s_0+s_5+s_6&s_1+s_4+s_5,s_1+s_3+s_6\\
\st{right}&s_0+s_1,s_2+s_3+s_4&s_1+s_4+s_5+s_6.
\end{array}
\]
Each term of~\eqref{eq:L6} that is not displayed is dominated by a displayed
sum, which gives the three costs. Their minimum is $F_7$, giving the asserted upper bound for $L_7^*$.
\end{proof}
Figure~\ref{fig:seven} shows chains attaining the costs.
\begin{table}[t]
\centering
\footnotesize
\renewcommand{\arraystretch}{1.05}
\begin{tabular}{c@{\qquad}llll}
\toprule
 & $(x_0,y_0)$ & $(x_1,y_1)$ & $(x_2,y_2)$ & $(x_3,y_3)$\\
strategy & $(x_4,y_4)$ & $(x_5,y_5)$ & $(x_6,y_6)$ & \\
\midrule
$\st{top}$ & $(0,\,0)$ & $(L-s_1,\,0)$ & $(L-s_2,\,s_1)$ & $(0,\,L-s_3)$ \\
 & $(s_3,\,L-s_4)$ & $(L-s_5,\,L-s_5)$ & $(s_3+s_4,\,L-s_6)$ &  \\[1.2pt]
$\st{center}$ & $(0,\,0)$ & $(L-s_1,\,0)$ & $(0,\,L-s_2)$ & $(s_2,\,L-s_3)$ \\
 & $(L-s_4,\,L-s_4)$ & $(L-s_5,\,s_1)$ & $(L-s_5-s_6,\,s_1)$ &  \\[1.2pt]
$\st{right}$ & $(0,\,0)$ & $(L-s_1,\,0)$ & $(0,\,L-s_2)$ & $(s_2,\,L-s_3)$ \\
 & $(L-s_4,\,L-s_4)$ & $(L-s_5,\,s_1)$ & $(L-s_6,\,s_1+s_5)$ &  \\[1.2pt]
\bottomrule
\end{tabular}
\medskip
\caption{Coordinates of the lower left corners of $Q_0,\dots,Q_6$ for the three strategies $\st{top}$, $\st{center}$, and $\st{right}$, as functions of the tuple $s$ and the side $L$.}
\label{tab:coords67}
\end{table}

\subsection{Cornering and a perpendicular chain cover}
For seven squares, a more general cornering argument replaces the
four-square windowpane. When every triple containing the largest square
clears $\Sigma$, we can move that square to a corner. Each remaining square must
then lie to its right or above it. Those two groups form chains in
perpendicular directions (Figure~\ref{fig:frame}), and the following lemmas show how to select
a pair of chains that accounts for all the remaining squares. This will
force a four-square chain in the proof of Proposition~\ref{prop:seven-lower}.
\begin{figure}[!htbp]\centering
\begin{minipage}[t]{0.45\textwidth}\centering
\begin{tikzpicture}[x=0.0895cm,y=0.0895cm]
\useasboundingbox (0,-5.13) rectangle (43.13,38);
\fill[vc!12] (22,0) rectangle (38,38);
\fill[hc!12] (0,22) rectangle (38,38);
\filldraw[fill=black!7,draw=black,line width=.35pt] (0,0) rectangle (22,22);
\node[font=\scriptsize,inner sep=0pt] at (11,11) {$0$};
\filldraw[fill=black!7,draw=black,line width=.35pt] (23,0) rectangle (38,15);
\node[font=\scriptsize,inner sep=0pt] at (30.5,7.5) {$1$};
\filldraw[fill=black!7,draw=black,line width=.35pt] (24,15) rectangle (38,29);
\node[font=\scriptsize,inner sep=0pt] at (31,22) {$2$};
\filldraw[fill=black!7,draw=black,line width=.35pt] (0,27) rectangle (11,38);
\node[font=\scriptsize,inner sep=0pt] at (5.5,32.5) {$3$};
\filldraw[fill=black!7,draw=black,line width=.35pt] (11,28) rectangle (21,38);
\node[font=\scriptsize,inner sep=0pt] at (16,33) {$4$};
\filldraw[fill=black!7,draw=black,line width=.35pt] (29,29) rectangle (38,38);
\node[font=\scriptsize,inner sep=0pt] at (33.5,33.5) {$5$};
\filldraw[fill=black!7,draw=black,line width=.35pt] (21,30) rectangle (29,38);
\node[font=\scriptsize,inner sep=0pt] at (25,34) {$6$};
\draw[line width=.7pt] (0,0) rectangle (38,38);
\filldraw[fill=hc!25,draw=hc!70!black,line width=.25pt] (0,-4.37) rectangle (11,-1.52);
\node[font=\tiny,inner sep=0pt] at (5.5,-2.945) {$3$};
\filldraw[fill=hc!25,draw=hc!70!black,line width=.25pt] (11,-4.37) rectangle (21,-1.52);
\node[font=\tiny,inner sep=0pt] at (16,-2.945) {$4$};
\filldraw[fill=hc!25,draw=hc!70!black,line width=.25pt] (21,-4.37) rectangle (29,-1.52);
\node[font=\tiny,inner sep=0pt] at (25,-2.945) {$6$};
\filldraw[fill=hc!25,draw=hc!70!black,line width=.25pt] (29,-4.37) rectangle (38,-1.52);
\node[font=\tiny,inner sep=0pt] at (33.5,-2.945) {$5$};
\filldraw[fill=vc!25,draw=vc!70!black,line width=.25pt] (39.52,0) rectangle (42.37,15);
\node[font=\tiny,inner sep=0pt] at (40.945,7.5) {$1$};
\filldraw[fill=vc!25,draw=vc!70!black,line width=.25pt] (39.52,15) rectangle (42.37,29);
\node[font=\tiny,inner sep=0pt] at (40.945,22) {$2$};
\filldraw[fill=vc!25,draw=vc!70!black,line width=.25pt] (39.52,29) rectangle (42.37,38);
\node[font=\tiny,inner sep=0pt] at (40.945,33.5) {$5$};
\end{tikzpicture}\par
\smallskip{\footnotesize corner frame, $L=38$\\ $s=(22,15,14,11,10,9,8)$}
\end{minipage}
\caption{The corner frame of Lemma~\ref{lem:full-frame} for the top
strategy, where every triple through $Q_0$ has side sum at least
$s_0+s_5+s_6=39>L$. The right arm $\{1,2,5\}$ lies in the purple region
$x\ge s_0$, the upper arm $\{3,4,5,6\}$ lies in the teal region $y\ge
s_0$, and the arms share $Q_5$.}
\label{fig:frame}
\end{figure}
\begin{lemma}[Corner frame]\label{lem:full-frame}
Fix a packing of $Q_0,\ldots,Q_N$ of side $L<\Sigma$ and an index
$m\le N$, and suppose every triple of indices in $\{0,\ldots,m\}$ that
contains $0$ clears $\Sigma$. Then $Q_0$ can be translated to a corner in the
packing of $Q_0,\ldots,Q_m$ obtained by removing $Q_{m+1},\ldots,Q_N$. After reflection,
let $Q_0=[0,s_0]^2$ and put
\[
\mathcal R=\{i\in\{1,\ldots,m\}:x_i\ge s_0\},\qquad
\mathcal U=\{i\in\{1,\ldots,m\}:y_i\ge s_0\}.
\]
The arms cover $\{1,\ldots,m\}$. The right arm is a vertical chain with
pairwise $x$-overlaps, the upper arm is a horizontal chain with
pairwise $y$-overlaps, and $|\mathcal R\cap\mathcal U|\le1$.
If $Q_0$ is already cornered in the packing of $Q_0,\ldots,Q_N$, these
conclusions about $Q_1,\ldots,Q_m$ require no motion.
\end{lemma}
\begin{proof}
Blockers on both sides of the horizontal band through $Q_0$ would form a
clearing horizontal triple with it. One side is therefore free; move
$Q_0$ to a vertical wall. Repeat vertically and reflect. Every other square
then lies right of $Q_0$ or above it. Two $x$-disjoint members of $\mathcal R$
would form a clearing horizontal triple with $Q_0$. They therefore
$x$-overlap and are $y$-disjoint. The upper-arm assertion is
symmetric. Two members in the intersection would be required both to
$x$-overlap, as right-arm members, and to be $x$-disjoint, as
upper-arm members.
\end{proof}

\begin{lemma}[Perpendicular covering chains]\label{lem:two-cover}
In the frame of Lemma~\ref{lem:full-frame}, with $m\ge1$, there are a vertical chain
$\mathcal V$ and a horizontal chain $\mathcal H$ such that
$\mathcal V\cup\mathcal H=\{1,\ldots,m\}$ and $\mathcal V\cap\mathcal H$
consists of one index. In particular, for $m=6$ some four squares form a
chain.
\end{lemma}
\begin{proof}
If one arm is empty, the other arm is a chain containing every index, and
we pair it with a singleton. If the arms meet, they are the required
chains. Otherwise the arms are disjoint and nonempty. Let $r$ be topmost in
$\mathcal R$ and $u$ rightmost in $\mathcal U$. Since
\[
x_r\ge s_0>x_u,\qquad y_u\ge s_0>y_r,
\]
either $r$ lies right of $u$, or $u$ lies above $r$. In the first case,
since $u$ is the rightmost member of $\mathcal U$ and lies left of $r$, every
member of $\mathcal U$ lies left of $r$. Thus, $\mathcal R$ and
$\mathcal U\cup\{r\}$ are the required chains. In the
second case, $\mathcal R\cup\{u\}$ and $\mathcal U$ are the required chains
by symmetry. The two chain cardinalities sum to $m+1$. Hence for $m=6$ one of
them is at least four.
\end{proof}

\subsection{Proof of the seven-square formula}
Put $W_{24}=s_2+s_4+s_5+s_6$.
\begin{proposition}[Three lower bounds]\label{prop:seven-lower}
For every $s\in\ell_2^\downarrow$,
\begin{align}
L_7^*(s)&\ge\min(C_0,W),\label{eq:seven1}\\
L_7^*(s)&\ge\min(C_3,R),\label{eq:seven2}\\
L_7^*(s)&\ge\min(C_0,H,W_{24}).\label{eq:seven3}
\end{align}
\end{proposition}
\begin{proof}
By Lemma~\ref{lem:positive}, we may assume $s_6>0$. For~\eqref{eq:seven1}, put $\Sigma=\min(C_0,W)$ and suppose a packing has
side $L<\Sigma$. Every triple through $Q_0$ has side sum at least $C_0$ and
clears $\Sigma$. Corner $Q_0$ and apply
Lemma~\ref{lem:two-cover} to the six remaining squares. The resulting
four-square chain has side sum at least $W$, a contradiction.
Inequality~\eqref{eq:seven2} is Proposition~\ref{prop:column-bound} with $n=7$.

For~\eqref{eq:seven3}, suppose $L<\Sigma:=\min(C_0,H,W_{24})$.
Corner $Q_0$ and form the corner frame on $Q_1,\ldots,Q_6$.
Every four-subset of these indices except $\{3,4,5,6\}$ has side sum
at least $W_{24}$ and clears $\Sigma$. Thus, the four-chain supplied by
Lemma~\ref{lem:two-cover} must be exactly $\{3,4,5,6\}$;
interchange axes to make it horizontal. At most one of its members lies
in the right arm. Hence at least three lie in the upper arm. Neither
$Q_1$ nor $Q_2$ can lie in the upper arm, because it and these three
would form a clearing horizontal quadruple. Thus, $Q_1,Q_2$ lie in the
exclusive right arm and form a vertical chain, whose upper member
$Q_t$ has $y_t<s_0$.

For $k=3,4,5$, the triple $\{Q_1,Q_2,Q_k\}$ has side sum at least $H$ and clears $\Sigma$. Hence $Q_k$ cannot lie in
the right arm. If $Q_k$ $x$-overlapped $Q_t$, it would lie
above $Q_t$, since $y_k\ge s_0>y_t$, and the squares $Q_1,Q_2,Q_k$
would form a clearing vertical chain. Hence $Q_3,Q_4,Q_5$ are all
$x$-disjoint from $Q_t$. They are mutually $x$-disjoint as members of
the forced four-chain. The quadruple $\{t,3,4,5\}$ differs from
$\{3,4,5,6\}$ and therefore clears $\Sigma$, a contradiction.
\end{proof}

\begin{theorem}\label{thm:seven}
For every $s\in\ell_2^\downarrow$,
\[
L_7^*(s)=F_7(s)=\min(\ell_{\st{top}},\ell_{\st{center}},\ell_{\st{right}}),
\]
where $\ell_{\st{top}},\ell_{\st{center}},\ell_{\st{right}}$ are the costs
in Proposition~\ref{prop:seven-upper}.
\end{theorem}
\begin{proof}
Proposition~\ref{prop:seven-upper} gives
$L_7^*\le F_7=\min(\ell_{\st{top}},\ell_{\st{center}},\ell_{\st{right}})$.
For the reverse inequality, note that $L_7^*\ge L_6^*$ by
Lemma~\ref{lem:function}. If $F_7=L_6^*$, this suffices. Otherwise
$F_7>L_6^*\ge s_0+s_1$, and
$F_7\le R$, $F_7\le\max(C_0,C_3)$, and $F_7\le\max(W,H)$.
If $C_3\ge F_7$, use~\eqref{eq:seven2}. Otherwise $C_0\ge F_7$.
If $W\ge F_7$, use~\eqref{eq:seven1}. Otherwise $H\ge F_7$.
Since $C_0\ge F_7>s_0+s_1$, we have $s_5+s_6>s_1$, whence
\[
W_{24}-H=s_4+s_6-s_1\ge s_5+s_6-s_1>0.
\]
All three terms in~\eqref{eq:seven3} are then at least $F_7$.
Thus, $L_7^*\ge F_7$.
\end{proof}

\section{Minimality of convex representations}\label{sec:minimality}
To end the paper, we prove that the complete strategy families of the preceding sections are minimal.
Every strategy cost $\ell_\sigma$ is a maximum of finitely many linear
functions, namely its chain sums, and is therefore convex.

\begin{lemma}[Counting convex costs]\label{lem:count}
Let $D$ be a convex cone and $F:D\to\mathbb R$ positively homogeneous.
Suppose $p_1,\ldots,p_N\in D$ satisfy
\begin{equation}\label{eq:pairwise-super}
F(p_i+p_j)>F(p_i)+F(p_j)\quad\text{for all }i<j.
\end{equation}
Every representation $F=\min_{1\le k\le m}h_k$ with convex $h_k$ on $D$
requires $m\ge N$.
\end{lemma}
\begin{proof}
Each $h_k\ge F$. For each $i$, choose $k_i$ such that
$h_{k_i}(p_i)=F(p_i)$. If $k_i=k_j=k$ for some $i<j$,
positive homogeneity of $F$ and convexity of $h_k$ give
\[
F(p_i+p_j)
=2F\left(\frac{p_i+p_j}{2}\right)
\le 2h_k\left(\frac{p_i+p_j}{2}\right)
\le F(p_i)+F(p_j),
\]
contradicting~\eqref{eq:pairwise-super}.
Thus, $k_1,\ldots,k_N$ are distinct, so $m\ge N$.
\end{proof}

\begin{theorem}\label{thm:min-counts}
For $n=6,7$, the least numbers $m$ of real-valued convex functions
$h_1,\ldots,h_m$ on $\ell_2^\downarrow$ with $L_n^*(s)=\min_kh_k(s)$ for every
$s\in\ell_2^\downarrow$ are $1$ and $3$, respectively.
These are also the minimum cardinalities of complete strategy families.
\end{theorem}
\begin{proof}
By Theorems~\ref{thm:six} and~\ref{thm:seven}, the strategy of
Proposition~\ref{prop:six-upper} and the three strategies of
Proposition~\ref{prop:seven-upper} form complete families of the stated
sizes, and their costs are convex. This proves the assertion for $n=6$.

For seven squares take
\[
p_1=(39,38,32,31,14,13,12),\quad p_2=(12,9,8,7,6,5,4),\quad
p_3=(16,10,9,8,7,6,5).
\]
Evaluation of~\eqref{eq:F7} gives
\[
\begin{array}{c|rrr}
s&p_1&p_2&p_3\\\hline
L_7^*(s)&77&21&26
\end{array}
\]
and
\[
L_7^*(p_1+p_2)=101>98,\quad
L_7^*(p_1+p_3)=104>103,\quad
L_7^*(p_2+p_3)=48>47.
\]
Apply Lemma~\ref{lem:count}.
\end{proof}

\begin{corollary}[Unit-container parameter sets]\label{rem:unit}
The set $X_6=\{s\in\ell_2^\downarrow:L_6^*(s)\le1\}$ is convex, whereas
$X_7=\{s\in\ell_2^\downarrow:L_7^*(s)\le1\}$ is a union of three convex
subsets and is not a union of two convex subsets.
\end{corollary}
\begin{proof}
The set $X_6$ is a sublevel set of the convex function $L_6^*$. By
Theorem~\ref{thm:seven}, $X_7$ is the union of the sublevel sets of the
three convex strategy costs $\ell_{\st{top}},\ell_{\st{center}},
\ell_{\st{right}}$. For the last assertion, let $p_1,p_2,p_3$ be the vectors
in the proof of Theorem~\ref{thm:min-counts}, and put
$\hat p_i=p_i/L_7^*(p_i)\in X_7$. For $i\ne j$, the point
$(p_i+p_j)/(L_7^*(p_i)+L_7^*(p_j))$ lies on the segment joining $\hat p_i$
and $\hat p_j$, and its optimal value
$L_7^*(p_i+p_j)/(L_7^*(p_i)+L_7^*(p_j))$ is greater than one. Thus, no
convex subset of $X_7$ contains two of $\hat p_1,\hat p_2,\hat p_3$, whereas
one of two sets covering $X_7$ would contain two of them.
\end{proof}
In the notation of~\cite{Cibulka17}, the points $\hat p_1,\hat p_2,\hat p_3$
form a clique in the invisibility graph of $X_7$, and the convexity number of
$X_7$ is three.

\section*{Development of the results and AI assistance}
I first read the paper of Moon and Moser in 2016. This project began nearly
ten years later, in June 2026, when I started actively looking for a formula
for $L_6^*$, and I obtained a complete proof of the formula for $L_7^*$ in
July 2026. Throughout this period I used Claude Opus~4.8 (Anthropic) as a
research assistant.

I understand a configuration of squares best by seeing it. In the second half
of June 2026, after finding the formula for $L_6^*$, I used Claude Opus~4.8 to develop a custom interactive tool for drawing and rearranging
axis-parallel squares, and this tool was an essential aid in the search for
$L_7^*$. I arranged squares by hand and read candidate terms off the chains
that determine the container side. The AI system computed exact optima for integer side vectors with the constraint-programming solver CP-SAT of Google OR-Tools and compared them with the candidate formulas. To find further arrangements, I examined the optimal packings at side vectors where a candidate formula failed. The strategies $\st{center}$ and $\st{right}$ in
Figure~\ref{fig:seven} arose by inserting $Q_6$ into the six-square arrangement
of Proposition~\ref{prop:six-upper}, and the $\st{top}$ strategy was read off optimal
packings at side vectors where those insertions failed.

I designed the lower-bound method: assume that a packing has side below the
asserted bound, use side sums exceeding that side to exclude chains, and follow
the forced overlaps to a contradiction. With the assistance of Claude
Opus~4.8, this approach was refined into its present form, built from clearing
sets, pins, cornering arguments, and a small number of lower bounds combined by
arithmetic. In preparing this manuscript, I used Claude Opus~5.5 (Anthropic)
and GPT-6 Astra Pro (OpenAI) to review the mathematical arguments and to improve the
exposition; Claude Opus~5.5 also assisted with typesetting, figures, and
exact-arithmetic checks of the constructions. I have checked all proofs and
take full responsibility for the content of this paper.

\newcommand{\etalchar}[1]{$^{#1}$}
\providecommand{\bysame}{\leavevmode\hbox to3em{\hrulefill}\thinspace}
\providecommand{\MR}{\relax\ifhmode\unskip\space\fi MR }
\providecommand{\MRhref}[2]{%
  \href{http://www.ams.org/mathscinet-getitem?mr=#1}{#2}
}
\providecommand{\href}[2]{#2}

\end{document}